\documentclass[reqno]{amsart}

\usepackage[english]{babel}

\title{Optimal eigenvalues on a metric graph with densities}

\usepackage{amsmath}
\usepackage{amsthm}
\usepackage{amssymb}
\usepackage{amsfonts}
\usepackage{bbm}
\usepackage{xcolor}
\usepackage{hyperref}
\usepackage{enumitem}
\usepackage{mathrsfs}  
\usepackage[normalem]{ulem}

\newcommand{\arxiv}[1]{\href{http://arxiv.org/abs/#1}{arXiv:#1}}

\newcommand*{\mailto}[1]{\href{mailto:#1}{\nolinkurl{#1}}}

\makeatletter
\@namedef{subjclassname@2020}{\textup{2020} Mathematics Subject Classification}
\makeatother

\newcommand{\msc}[1]{\href{http://www.ams.org/msc/msc2020.html?t=&s=#1}{#1}}

\theoremstyle{plain}
\newtheorem{thm}{Theorem}[section]
\newtheorem{lem}[thm]{Lemma}
\newtheorem{cor}[thm]{Corollary}
\newtheorem{prop}[thm]{Proposition}
 \newtheorem{conjecture}{Conjecture}

\theoremstyle{definition}
\newtheorem{defn}[thm]{Definition}

\theoremstyle{remark}
\newtheorem{rmk}[thm]{Remark}
\newtheorem{ex}[thm]{Example}

\newcommand\cI{{\mathcal{I}}}

\numberwithin{equation}{section}

\newcommand{\N}{\mathbb{N}}     
\newcommand{\R}{\mathbb{R}}     

\newcommand{\cM}{\mathcal{M}}  

\newcommand{\cN}{\mathcal{N}}

\newcommand{\cP}{\mathcal{P}}   
\newcommand{\one}{\mathbbm{1}}

\newcommand{\supp}{\operatorname{supp}}

\newcommand{\dom}{\operatorname{dom}}
\newcommand{\diam}{\operatorname{diam}}
\newcommand{\Span}{\operatorname{span}}
\newcommand{\AC}{\operatorname{AC}}

\newcommand{\ack}{\section*{Acknowledgments}}

\begin{document}

\author[K. Naderi]{Kiyan Naderi}
\address{Institute of Mathematics\\ University of Innsbruck\\
	Technikerstraße 13\\ 6020 Innsbruck\\ Austria}
\email{\mailto{kiyan.naderi@uibk.ac.at}}

\author[N. Nicolussi]{Noema Nicolussi}
\address{Institute of Mathematics\\ University of Innsbruck\\
	Technikerstraße 13\\ 6020 Innsbruck\\ Austria}
\email{\mailto{noema.nicolussi@uibk.ac.at}}


\thanks{{\it This is the author-accepted version of the article published in Journal of the London Mathematical Society (DOI: 10.1112/jlms.70714). For open access purposes, the authors have applied a CC BY public copyright license to this author-accepted version. Both authors were supported by the Austrian Science Fund (FWF) through project 10.55776/PAT4632823.}}

\maketitle

\begin{abstract}
	We introduce and study Laplacians on a finite metric graph endowed with generalized densities, that is, measures of finite mass. One important motivation is that this setting provides a common framework for several interesting classes of operators: discrete graph Laplacians, Kirchhoff Laplacians and Dirichlet-to-Neumann operators on graphs.

	Our main interest lies in spectral optimization with respect to the underlying measure. In contrast to the setting of domains and manifolds, we prove that a minimal $k$-th eigenvalue exists, whereas the corresponding maximization problem has no meaning. We then establish connections between these optimal eigenvalues and the geometry of the metric graph, including a transparent geometric characterization of the first optimal eigenvalue via the resistance metric, and a Weyl law for the higher optimal eigenvalues.

\end{abstract}

\section{Introduction}
Let $X$ be a Euclidean domain or Riemannian manifold and consider the eigenvalue problem
\begin{equation} \label{eq:EigenvalueEquation}
	\Delta f = \lambda f  \mu,
\end{equation}
where $\Delta$ is the Laplacian, $\lambda$ the spectral parameter and $\mu$ a suitable positive measure on $X$. A natural question is to optimize the eigenvalues of~\eqref{eq:EigenvalueEquation} when $\mu$ varies in a fixed class of measures. Such problems have a long tradition and we only point out some variants which have been studied.

M.\ G.\ Krein considered~\eqref{eq:EigenvalueEquation} for the Dirichlet Laplacian on a compact interval and measures of fixed mass with an $L^\infty$-density satisfying upper and lower bounds~\cite{krein}. He was able to explicitly compute the densities which minimize and maximize the $k$-th eigenvalue. The analogous problem for higher dimensional domains is sometimes referred to as inhomogeneous or composite media (or membranes in dimension $n=2$) and has received much attention, see the overview in~\cite[Chapter 9]{henrot} and~\cite{cgiko, cmI, cmII, friedland}.

In the setting of Riemannian manifolds, the optimization problem appeared recently for measures with smooth densities in~\cite{ce19, ces15, ksII, vinokurov}. For Riemannian surfaces, the corresponding maximization problem coincides with the one for conformal eigenvalues due to conformal invariance of the energy form~\cite{ce03}. Motivated by regularity questions on extremal metrics for conformal eigenvalues, Kokarev introduced a framework for eigenvalue optimization with respect to measures on Riemannian surfaces~\cite{kokarev}. Conceptually, he regarded measures as singular metrics and investigated, among others, existence and regularity questions for optimal measures. Bounds for eigenvalues with respect to measures and regularity of optimal measures were investigated also more recently in~\cite{ksI}. On higher dimensional manifolds, eigenvalues associated to general measures were also considered in~\cite{gkl}, where they lead to a common framework for Laplace and Steklov eigenvalues and several applications in shape optimization.

\smallskip

This paper concerns the above optimization problem on a compact metric graph~$G$ (i.e., a finite graph whose edges are considered as intervals of certain lengths). The first task consists in giving a meaning to the eigenvalue equation~\eqref{eq:EigenvalueEquation} on $G$. To this end, to each finite Borel measure $\mu$ on $G$, we associate a Laplacian
\[
	H_\mu \colon \dom(H_\mu) \subseteq L^2(G, \mu) \to L^2(G, \mu)
\]
which is a self-adjoint operator in the $L^2$-space $L^2(G, \mu)$. The construction is based on considering the energy form
			\[
				q(f) =\int_G |f'(x)|^2 \, dx,
			\]
			where $dx$ denotes the Lebesgue measure on $G$, in the Hilbert space $L^2(G, \mu)$ of the measure $\mu$ (see Section~\ref{sec:LaplaciansMeasures} for details). We stress that the energy $q(f)$ of a function $f$ is independent of the chosen measure $\mu$ and, thus, changing $\mu$ essentially amounts to considering the same quadratic form in different Hilbert spaces $L^2(G, \mu)$ (up to choosing a correct form domain).

It turns out that several well-studied classes of operators fall into this framework (see Section~\ref{ss:LaplacianBasics}). If the support of $\mu$ coincides with the vertex set $V$ of $G$, then the operator $H_\mu$ is a weighted discrete Laplacian (Example~\ref{ex:DiscreteLaplacians}). On the other hand, if $\mu$ has a positive density with respect to the Lebesgue measure $dx$ on $G$, then $H_\mu$ turns out to be a weighted Kirchhoff Laplacian (a.k.a. quantum graph~\cite{BerkolaikoKuchment, kurasov2023spectral}, see Example~\ref{ex:WeightedKirchhoffLaplacian}). Discrete Laplacians and Kirchhoff Laplacians are known to enjoy non-trivial relations~\cite{bgp, ekmn, konic21, pankrashkin, VonBelow} and, thus, putting them into the above common framework is also conceptually interesting. As another special case, weighted Dirichlet-to-Neumann operators on $G$ correspond to measures $\mu$ supported on a proper subset of the vertex set $V$ (Example~\ref{ex:DtN}).

For a general measure $\mu$, the Laplacian $H_\mu$ acts on each interval edge as a measure-derivative $f \mapsto - \frac{d}{d\mu} \frac{d}{dx} f$ (see Appendix~\ref{sec:SturmOscillationTheorem}). In the special case of only one interval, such generalized differential operators are also known as Krein strings~\cite{kk}.

\smallskip

Each Laplacian $H_\mu$ has non-negative, purely discrete spectrum. Counting multiplicities, we enumerate its non-zero eigenvalues as an infinite, increasing sequence $(\lambda_k (H_\mu))_{k=1}^\infty$, with convention $\lambda_k(H_\mu) := + \infty$ for $k \ge \dim(L^2(G, \mu))$.

\smallskip

Our main interest is then to optimize the eigenvalue $\lambda_k(H_\mu)$, $k \ge1$, with respect to the measure $\mu$. For obvious scaling reasons, we impose the normalization condition $\mu(G)=1$, that is, we restrict to the class $\cN$ of probability measures on $G$. It turns out that the maximization problem is trivial in the sense that (Corollary~\ref{cor:MaximumMeaningless})
\[
	\sup_{\mu \in \widetilde \cN} \lambda_k(H_\mu) = +\infty
\]
for all weak-$\ast$-dense sets of measures $\widetilde \cN \subset \cN$. However, the minimization problem is non-trivial with
\[
	\inf_{\mu \in \widetilde \cN} \lambda_k(H_\mu)  = \min_{\mu \in \cN} \lambda_k(H_\mu) > 0
\]
for every weak-$\ast$-dense set $\widetilde \cN \subset \cN$. The above claims follow from a continuity result (Theorem~\ref{thm: spec conv }), which asserts that the $k$-th eigenvalue $\lambda_k(H_\mu)$ depends continuously on $\mu$ in the weak-$\ast$ topology.

Let us stress that this behavior is actually opposite to the situation for higher dimensional manifolds. Indeed, in dimension $n \ge 2$, the $k$-th eigenvalue on the weak-$\ast$-dense class of probability measures with smooth, positive densities remains uniformly bounded above, but can be made arbitrarily small~\cite{ce19, ces15, gny}. On the other hand, it was already observed in \cite{cp18} that the behavior in dimension $n=1$ is opposite. One explanation is that, in higher dimensions, the dependence of the eigenvalues $\lambda_k(\mu)$ on $\mu$ is only upper semi-continuous, see e.g.~\cite[Proposition 1.1]{kokarev} and~\cite[Proposition 4.1]{gkl}.

\smallskip

Motivated by the above, we define the {\em $k$-th optimal eigenvalue} of a compact metric graph $G$ as
\[
	\lambda_k^{\min}(G) := \min_{\mu \in \cN} \lambda_k(H_\mu).
\]
This paper is devoted to the study of the optimal eigenvalues $\lambda_k^{\min}(G)$, $k \in \N$.

Let us stress the following physical interpretation of $\lambda_k^{\min}(G)$. The measure $\mu$ can be interpreted as a mass distribution in a weighted wave equation
\[
	\mu \cdot \partial_{tt}u(x,t) = \Delta u(x,t).
\]
After separating the variables, we obtain precisely the eigenvalue equation~\eqref{eq:EigenvalueEquation}. Thus, the above minimization problem is equivalent to distributing the mass on an oscillating graph in order to obtain the smallest possible $k$-th natural frequency.

Optimization problems for eigenvalues of Laplace operators on metric graphs have been studied in several settings. For example, in \cite{Band}, Band and Lévy investigated how to optimize the first nontrivial eigenvalue of the Kirchhoff Laplacian while varying the edge lengths on a fixed combinatorial graph. This essentially means searching for an optimal metric on a fixed topological structure. In contrast, in the present paper, we fix both the combinatorial graph and the edge lengths and consider optimization of the eigenvalues with respect to the measure defining the operator. Instead of an optimal geometric structure, we seek an optimal Laplace operator. To our knowledge, this question has not yet been considered. We also mention the recent preprint~\cite{Cao}, where extremal (or critical) points of the eigenvalue functional for graph Laplace operators depending simultaneously on a metric and a smooth measure are studied. To a certain extent, this combines the approach in \cite{Band} with ours, although we stress that the restriction to smooth measures from~\cite{Cao} is indeed significant and excludes the optimal measures obtained in our paper.

\subsection*{The first optimal eigenvalue $\lambda_1^{\min}(G)$} The optimal eigenvalue $\lambda_1^{\min}(G)$ displays several interesting properties. For instance, in Section~\ref{sec:FirstEigenvalue} we derive three different characterizations of $\lambda_1^{\min}(G)$. In the following, we focus on the first characterization (Theorem~\ref{Lem: Resitance metric and optimal eigenvalue}). Note that, for both discrete and metric graph Laplacians, path metrics on graphs and their diameters play an important role in spectral estimates, see e.g.~\cite{chung, grigoryanbook, kkmm}. As we show,  $\lambda_1^{\min}(G)$ is closely related to a different metric: the resistance metric $r$ (see Section~\ref{ss:CharacterizationResistance}), where the distance $r(x,y)$ between two points $x,y \in G$ is the resistance between them in the sense of electrical networks. Our first main result, Theorem~\ref{Lem: Resitance metric and optimal eigenvalue}, provides an explicit geometric characterization of $\lambda_1^{\min}(G)$:
\[
	\lambda_1^{\min}(G) = \frac{4}{\diam_r(G)},
\]
where $\diam_r(G)$ is the diameter of $G$ with respect to the resistance metric $r$. Moreover, the minimizing measures are those of the form $\mu = (\delta_x + \delta_y)/2$ for points $x,y\in G$ realizing the $r$-diameter.
In particular, the minimizing Laplacians are Dirichlet-to-Neumann operators (see Remark~\ref{rem:DtNGeneral}).

The proof of Theorem~\ref{Lem: Resitance metric and optimal eigenvalue} is based on a second characterization of $\lambda_1^{\min}(G)$ from Theorem~\ref{thm: characerisation of lambda1} in terms of certain discrete measures. Theorem~\ref{thm:SpectralPartitioning} provides a third characterization related to graph partitioning. The idea of incorporating graph partitions into the analysis originates mainly from \cite{SpecPart}. For further work on spectral partitioning with respect to the Kirchhoff Laplacian, see also~\cite{Nodalcountandpartition,SpecPart2}.

\smallskip

In Section~\ref{sec:PropertiesFirstEigenvalue}, we study further properties of $\lambda_1^{\min}(G)$. We derive estimates for $\lambda_1^{\min}(G)$ in terms of the natural path metric on $G$ and a Cheeger-type constant, and investigate their optimality. Moreover, Theorem~\ref{thm: points avoid vertices} asserts that the points $x$ and $y$ supporting a minimizing measure cannot be vertices of degree $\ge 3$. Finally, we compute $\lambda_1^{\min}(G)$ for several examples.

\subsection*{The higher optimal eigenvalues and Weyl's law}
It turns out that the higher optimal eigenvalues $\lambda_k^{\min}(G)$, $ k \ge 2$, are difficult to find even in simple cases. Whereas we compute them for the path graph $P = [0,1]$ (Theorem~\ref{thm:ConfEVPath}), even this most basic example requires non-trivial considerations and relies on nodal properties of eigenfunctions, namely a Sturm oscillation theorem on $P$ (Theorem~\ref{Thm: Sturm oscillation}). Note that for the smaller class of probability measures on $P$ having a density satisfying upper and lower bounds, the maximization and minimization problems for the $k$-th eigenvalue were both solved by Krein~\cite{krein}.

On the other hand, for a general graph $G$, the asymptotic behavior of the optimal eigenvalues for $k \to  \infty$ becomes transparent. As our second main result, we prove a Weyl law (Theorem~\ref{thm:Weyl'sLaw}), showing that
\begin{equation} \label{WeylIntro}
	\frac{\lambda_k^{\min}(G)}{k^2} \to \frac{4 }{L(G)}, \qquad k \to \infty,
\end{equation}
where $L(G)$ is the total length of $G$.

It should be noted that for every fixed measure $\mu \in\cN$ on $G$, it follows from classical results that the Laplacian $H_\mu$ obeys the Weyl law
\begin{equation} \label{WeylIntro2}
	\frac{\lambda_k  (H_\mu)}{k^2} \to \frac{\pi^2 }{\big(\int_G \sqrt{ f } \, dx\big)^2  }, \qquad k \to \infty,
\end{equation}
where $f =\frac{d\mu_{\mathrm{ac}}}{dx}$ is the density of the absolutely continuous part $\mu_{\mathrm{ac}}$ of $\mu$. Indeed, this is stated for a single interval in~\cite[Chapter 6.8]{gk}, and one can pass to a general graph $G$ by arguments as in the proof of Lemma~\ref{lem:EigenvalueInterlacingConsequence}, see also~\cite{BerkolaikoKuchment, kurasov2023spectral} for the special case that $\mu$ is the Lebesgue measure $dx$ on $G$ (the corresponding Laplacian is the standard unweighted Kirchhoff Laplacian).

However, we stress that our Weyl law~\eqref{WeylIntro} is conceptually rather different from~\eqref{WeylIntro2}. First of all, the factors $\pi^2$ and $4$ differ. Moreover, as the example of the path graph shows, the optimal eigenvalues in general stem from infinitely many pairwise different measures, which are moreover all finitely supported. In particular, each corresponding operator has only finitely many eigenvalues and its Weyl law~\eqref{WeylIntro2} thus has no proper meaning. On the other hand,~\eqref{WeylIntro} provides a non-trivial asymptotical statement about these infinitely many different operators.

\subsection*{Structure of the paper}
Section~\ref{sec:Preliminaries} contains preliminaries on metric graphs and their function spaces. In Section~\ref{sec:LaplaciansMeasures}, we introduce the Laplacian $H_\mu$ for a measure $\mu$ on $G$ and establish its basic properties. The optimal eigenvalues are introduced in Section~\ref{sec:DefinitionEigenvalues}. Section~\ref{sec:FirstEigenvalue} contains different characterizations of $\lambda_1^{\min}(G)$. Further properties of $\lambda_1^{\min}(G)$ are investigated in Section~\ref{sec:PropertiesFirstEigenvalue}. Section~\ref{sec:Weyl'sLaw} contains our results on the higher optimal eigenvalues. Finally, Appendix~\ref{sec:SturmOscillationTheorem} provides a Sturm oscillation theorem for Laplacians on path graphs.


\section{Preliminaries} \label{sec:Preliminaries}
This section is of preliminary character, where we introduce metric graphs, their Sobolev spaces and their harmonic functions.

\subsection{Metric graphs and their Sobolev spaces} \label{ss:MetricGraphsBasics}

In what follows, $G_d = (V, E)$ will be an unoriented, finite and connected graph with vertex set $V$ and edge set $E$. For two vertices $u$, $v\in V$ we shall write $u\sim v$ if there is an edge $e \in E$ connecting $u$ and $v$. Conversely, if $e$ is an edge connecting $u$ and $v$, then, with a slight abuse of notation, we write $e=uv$, even if there may be several such edges. If an edge $e \in E$ is incident to a vertex $v \in V$, then we also write $e \sim v$. The \emph{degree} $\deg(v)$ of a vertex $v \in V$ is defined as the number of edges incident to $v$ (with loop edges counted twice).

Assigning each edge $e\in E$ a finite length $\ell(e) \in (0,\infty)$ and considering the corresponding {\em edge length function} $\ell \colon E\to(0, \infty)$, we can then naturally associate with $(G_d,\ell) = (V, E, \ell)$ a metric space $G$: we first define an orientation on the edge set $E$ and then identify each edge $e$ with a copy of the interval $\cI_e = [0, \ell(e)]$, with the initial and terminal vertex corresponding to the left, respectively, right endpoint of $\cI_e$. The topological space $G$ is then obtained by ``glueing together" the endpoints of edges corresponding to the same vertex (in the sense of a topological quotient). $G$ is a compact Hausdorff space and metrizable by the {\em path metric}\label{not:natpathmet} $\varrho$: the distance $\varrho(x,y)$ between two points $x,y \in G$ is defined as the arc length of the shortest path connecting them.

A \emph{metric graph} is a metric space $G$ arising from the above construction for some collection $(G_d, \ell) =(V, E,\ell)$. More specifically, $G$ is then called the \emph{metric realization} of $(G_d, \ell)$. Conversely, a pair $(G_d, \ell)$ whose metric realization coincides with $G$ is called a \emph{model} of $G$.

Clearly, different models may give rise to the same metric graph. Moreover, any metric graph $G$ has infinitely many models. E.g., different models can be constructed by {\em subdivision}, i.e., subdividing an edge $e$ of $G$ into two edges $e_1$ and $e_2$ and equipping them with lengths such that $\ell(e_1)  + \ell(e_2)= \ell(e)$.

A graph $G_d = (V,E)$ is called \emph{essential} if either it has no vertices of degree two, or it is the graph consisting of one vertex and one loop edge. Any metric graph $G$ has a model $(G_d, \ell)$ in which $G_d$ is essential. Such a model is called an \emph{essential model} for $G$. All essential models of $G$ have the same underlying essential graph $G_d$. Moreover, the essential model $(G_d, \ell)$ for a metric graph $G$ is unique unless $G$ has non-trivial automorphisms, in which case, $G$ has a finite number of essential models $(G_d, \ell)$. (These are obtained by applying the automorphisms $\varphi$ to one fixed model $(G_d, \ell)$, that is, considering edge length functions given by compositions $\ell \circ \varphi$.)

In our paper, we will usually consider a metric graph together with a fixed, not necessarily essential model. Abusing slightly the notation, we will usually not distinguish between the two objects and also write $G = (V, E, \ell)$ or $G= (G_d, \ell)$.

To extend the notion of degree to arbitrary points $x \in G$, we set $\deg(x) := 2$, whenever $x \notin V$ for at least one model $(G_d, \ell) = (V,E,\ell) $.

For a metric graph $G = (V,E, \ell)$, we denote by $L= L(G)=\sum_{e\in E} \ell(e)$ its {\em total length} and by
\[
	\diam(G) := \diam_\varrho(G) := \max_{x,y \in G} \varrho(x,y)
\]
its {\em diameter} with respect to the path metric $\varrho$. Note also that every connected component $G'$ of a closed subset of $G$ is again a connected metric graph.

We denote by $dx$ the natural Lebesgue measure on $G$. The Lebesgue volume of a measurable subset $A\subset G$ is denoted by $L(A)$. The Lebesgue $L^2$-function space is then given by
\[
	L^2(G):= L^2(G, dx) := \left\{ f\colon G\to \R \text{ measurable } \; | \; \|f\|^2_{L^2(G)} := \int_{G}|f|^2 \, dx < \infty \right\}.
\]
Furthermore, we define the Sobolev space $H^1(G)$ on $G$ as
\[
	H^1(G) := \left\{ f\colon G\to \R \; | \; f \text{ is continuous, } \left.f\right|_{e} \in H^1((0,\ell(e)))  \text{ for every }e \in E \right\},
\]
where $H^1((0,\ell(e)))$ is the standard first-order Sobolev space on the interval $(0, \ell(e))$. The \emph{energy form} $q$ is defined as
\begin{align*}
	q\colon & H^1(G) \times H^1(G) \to \R, \qquad q(f,g) := \int_{G} f'g' d x.
\end{align*}
We abbreviate $q(f) := q(f,f)$, $f \in H^1(G)$, and call $q(f)$ the {\em energy of $f$}.

\begin{rmk}
	The derivative of a function $f \in L^2(G)$ on an edge $e$ is always meant in the distributional sense, i.e., with regard to test functions in $C^\infty_0(e) = C^\infty_0((0, \ell(e)))$.
\end{rmk}
\begin{rmk}
	Note that the definition of the derivative needs a priori a choice of orientation on the edges. However, for the quantities which are of interest to us, a change of orientation results at most in a change of sign. For instance, the energy form $q$ and scalar product $(\cdot , \cdot)_{H^1}$ are independent of the choice of orientation.
\end{rmk}
The Sobolev space $H^1(G)$ endowed with the scalar product
\[
	(f,g)_{H^1} := (f,g)_{H^1(G)} :=  \int_G fg \, d x + \int_{G} f'g' d x
\]
is well-known to be a Hilbert space. The corresponding norm is denoted by $\|f \|_{H^1} := (f,f)_{H^1}^\frac{1}{2}$, $f \in H^1(G)$.

For open subsets $O \subset G$, the spaces $L^2(O)$ and $H^1(O)$ are defined via restrictions,
\[
	H^1(O) := \{f \colon O \to \R  \; | \; \exists \widetilde{f} \in H^1(G)\text{ with } \widetilde{f} \big|_{O} \equiv f\}, \quad
\]
and
\[
	\|f\|_{H^1(O)} := \inf_{\substack{\widetilde{f} \in H^1(G), \; \widetilde{f} \big|_{O} \equiv f}} \|\widetilde{f} \|_{H^1(G)}.
\]
For closed subsets $A \subset G$ with connected components $A_i \subset A$, the spaces $L^2(A)$ and $H^1(A)$ are defined as the direct sum of the respective spaces on every connected component $A_i$ considered as its own metric graph.

We often split a function $f \in H^1(G)$ into the family of functions $f_e := f\big|_{\mathring{e}}$, $e \in E$, where $\mathring{e}$ means the interior of the interval $e\simeq [0, \ell(e)]$.

\subsection{Harmonic functions on metric graphs}
Let $G= (V,E, \ell)$ be a metric graph.
\begin{defn}
	Let $O\subset G$ be open. A function $f \in H^1(G)$ is called harmonic on $O$, if
	\[
		q(f,g) = 0 \qquad \text{ for every } g \in H^1(G) \text{ with } g|_{G \setminus O} \equiv 0.
	\]
\end{defn}

Harmonic functions admit the following pointwise description.

\begin{lem} \label{Prop: characterization harmonicity} Let $O \subset G$ be open and $f \in H^1(G)$. The following statements are equivalent:
	\begin{itemize}
		\item [(i)]  $f$ is harmonic on $O$.
		\item [(ii)] In every point $x \in O$, the following holds:
		      \begin{itemize}
			      \item If $x \in O \setminus V$ is not a vertex, then there exists an open neighborhood $U(x) \subset O$ of $x$ which is isomorphic to an interval $(a,b)$ and such that $f$ is linear on $U(x)$.
			      \item If $x \in O \cap V$ is a vertex, then $0 = \sum_{e \in E, e \sim x } f_e'(x)$, where all edges $e$ are oriented towards $x$ in the computation of $f_e'$.
		      \end{itemize}
	\end{itemize}
\end{lem}
\begin{proof}
	Assume that $f  \in H^1(G)$ is harmonic. Let $x \in  O$ be a point in the interior of an edge $e \in E$. Consider a test function $g$ supported in a small open neighborhood $U(x) \subset O \cap e$. Then
	\[
		0= q(f,g)= \int_{e} f' g' d x = - \int_{e} f'' g d x,
	\]
	so $f \big|_{U(x)}'' \equiv 0 \in L^2(U(x))$ and $f$ is linear on $U(x)$.

	On the other hand, if $x \in O \cap V$ is a vertex, then we can choose a test function $g \in H^1(G)$ supported in a small open neighborhood $U(x) \subset O$ of $x$ with $g(x) \neq 0$. Using the first observation we obtain
	\begin{align*}
		0 = q(f,g)  = \sum_{e: e\sim x} \int_{e} f_e' g_e' d x & = \sum_{e: e \sim x } g(x)f_{e}'(x) - \int_{e} g f_{e}'' d x \\&= g(x)\sum_{e: e \sim x} f_{e}'(x) ,
	\end{align*}
	where all edges $e$ are  oriented towards $x$. Altogether, every harmonic function $f \in H^1(G)$ satisfies the above pointwise conditions. Conversely, if the pointwise conditions hold for $f \in H^1(G)$, then the harmonicity of $f$ follows from a similar argument using integration by parts.
\end{proof}

In the following, we establish basic facts about harmonic functions on metric graphs. Note that analogous results for weighted discrete graphs are well-known, see e.g.~\cite{grigoryanbook}. Whereas essentially the proofs carry over to metric graphs, we include them for the sake of completeness. We begin with the maximum principle.

\begin{prop}\label{Prop: Max and Min in harmonic functions}
	Let $O \subset G$ be open and $f \in H^1(G)$ harmonic on $O$. Then
	\[
		\max_{x \in G} f(x) = \max_{x \in G\setminus O} f(x), \qquad \min_{x \in G} f(x) = \min_{x \in G\setminus O} f(x).
	\]
\end{prop}
\begin{proof}
	Since $f$ is continuous and $G$ is compact, $f$ attains its maximum. Assume that $\max_{x \in G\setminus O} f(x) <  \max_{x \in G} f(x)$ and fix a point $x_0 \in O$ with $f(x_0) = \max_{x \in G} f(x)$. Lemma~\ref{Prop: characterization harmonicity} implies that $f$ is constant on the whole connected component $O'\subset O$ containing $x_0$. Thus, there exists $y \in \partial O' \subset G \setminus O$ with $f(y)=\max_{x\in G} f(x)$, which is a contradiction. The same applies to the minimum of $f$.
\end{proof}

We have the following results on existence and uniqueness of solutions to Dirichlet problems on metric graphs.

\begin{prop} \label{Prop: Harmonic extension}
	Let $\varnothing \neq A \subset V$ be finite and let $(f_x)_{x \in A} \subset \R$ be given function values.
	\begin{enumerate}
		\item [(i)] \label{1:HarmonicExtension} There exists a unique function $f \in H^1(G)$ such that
		      \begin{equation} \label{harmonic extension}
			      \begin{cases}
				      f(x) = f_x \text{ for all } x \in A, \\
				      f \text{ is harmonic on } G\setminus A.
			      \end{cases}
		      \end{equation}

		\item [(ii)] \label{2:HarmonicExtension} The solution of~\eqref{harmonic extension} is the unique function $f \in H^1(G)$ such that
		      \begin{equation} \label{harmonic extension minimum}
			      \begin{cases}
				      f(x) = f_x,   & \text{ for all } x \in A,            \\
				      Q(f)\le Q(g), & \text{ for all }  g \in \mathcal{F},
			      \end{cases}
		      \end{equation}
		      where $\mathcal{F} = \{g \in H^1(G)| \, g(x) = f_x \text{ for all } x \in A\}$.
		\item [(iii)] \label{3:HarmonicExtension} Let $f\in H^1(G)$ be harmonic on $G \setminus A$ and $g \in H^1(G)$. Then
		      \begin{equation} \label{harmonic extension IBP}
			      q(g,f) =  \sum_{x \in A}g(x) \sum_{e: e \sim x } \frac{f(x) - f(y_e)}{\ell(e)},
		      \end{equation}
		      where $y_e \in V$ is the unique vertex which is connected to $x$ via $e$.
	\end{enumerate}
\end{prop}

\begin{proof}
	We first prove that solutions to~\eqref{harmonic extension} are unique. Let $f_1,f_2 \in H^1(G)$ be two solutions of \eqref{harmonic extension} and consider $f:=f_1-f_2$. Then $f$ is harmonic on $G \setminus A$ and $f(x)=0$ for all $x \in A$, so $f \equiv 0$ by Proposition~\ref{Prop: Max and Min in harmonic functions}.

	As  the next step, we prove that every solution $f$ to~\eqref{harmonic extension minimum} solves~\eqref{harmonic extension}. Let $g \in H^1(G)$ with $g\big|_{A} \equiv 0$. For every $t \in \R$, the function $f_t := f+tg$ satisfies $f_t(x) = f_x$ for all $x \in A$. We infer that
	\[
		0 = \frac{d}{dt}\Big|_{t=0} q(f+tg) = 2q(f,g).
	\]
	This proves that $f$ is harmonic on $G\setminus A$ and solves~\eqref{harmonic extension}.

	The claims in~(i) and~(ii) then follow, if we prove that~\eqref{harmonic extension minimum} has a solution $f \in H^1(G)$. Define $F \in H^1(G)$ as the edgewise linear function with values in vertices
	\[
		F(x) := \begin{cases}
			f_x, & x \in A,             \\
			0,   & x \in V \setminus A.
		\end{cases}
	\]
	Consider the following finite dimensional function space
	\[\mathcal{F}_1 := \{f \in H^1(G) \; | \; f \text{ linear on every edge}, f(x) = 0 \text{ for all } x \in A \}\]
	and also
	\[\mathcal{F}_2 := \{f \in H^1(G) \; | \; f(x) = 0 \text{ for all } x \in A\}.\]
	By construction, $\mathcal{F} = F + \mathcal{F}_2$.

	Fix $x_0 \in A$. Every $x \in G$ can be connected to $x_0$ by a path $\gamma$ of length $\ell(\gamma)\le \diam G$. Then for every $f \in F + \mathcal{F}_2$ holds
	\begin{align*}
		|f(x)| & \leq \int_{\gamma} |f'(s)| d s + |f_{x_0}| \\ &\leq \sqrt{\int_{\gamma} |f'(s)|^2 d s} \sqrt{\ell(\gamma)} + |f_{x_0}|  \leq  \sqrt{\diam G} \sqrt{q(f)}  + |f_{x_0}| .
	\end{align*}

	Integrating over $G$ leads to
	\[
		\|f\|_{L^2(G)} \leq  \sqrt{L(G)} \|f\|_{\infty}  \leq  \sqrt{L(G)}( \sqrt{\diam G} \sqrt{q(f)} + |f_{x_0}|).
	\]
	Thus, if a sequence $(f_n)_{n \in \N} \subset F +\mathcal{F}_1$ tends to infinity in $\|\cdot\|_{H^1}$-norm, then $q(f_n) \to \infty $. In addition, $q\colon H^1(G) \to \R_{\geq 0}$ is continuous and $F + \mathcal{F}_1$ is closed, so $q$ must attain its minimum in $F+\mathcal{F}_1$.

	Let $f$ be a minimizer of $q$ in $F+\mathcal{F}_1$. Then $f$ minimizes $q$ also in $\mathcal{F} = F+\mathcal{F}_2$. Indeed, assume there exists $g \in \mathcal{F}_2$ with $q(F+g) < q(f)$. Write $g=g_1 +g_2$ with $g_1 \in \mathcal{F}_1$ and $g_2 \in H^1(G)$ with $g_2(x) =0$ for every vertex $x \in V$. Using Lemma~\ref{Prop: characterization harmonicity} we infer that $F+g_1$ is harmonic on $G \setminus V$ and thus $q(F+g_1, g_2) = 0$.
	But then
	\begin{align*}
		q(f) > q(F+g) = q(F+g_1) + q(g_2) \geq q(F+g_1),
	\end{align*}
	contradicting that $f$ minimizes $q$ in $F+\mathcal{F}_1$. Thus, $f$ satisfies~\eqref{harmonic extension minimum}.

	It remains to prove the equality in~\eqref{harmonic extension IBP}. Let $f, g \in H^1(G)$ be as above. Lemma~\ref{Prop: characterization harmonicity} shows that $f$ is linear (or harmonic) on every edge. We decompose $g=g_1+g_2$ with $g_1$ linear on every edge, $g_1(x)=g(x)$ for every $x \in A$, and $g_1(x)=0$ for every $x \in V \setminus A$. Then $g_2\big|_{A} \equiv 0$ and we calculate
	\begin{align*}
		q(g,f) = q(g_1, f) = \sum_{e \in E} \int_{e} g_1' f' d x = \sum_{x \in A}g(x) \sum_{e: e \sim x } \frac{f(x) - f(y_e)}{\ell(e)},
	\end{align*}
	completing the proof.
\end{proof}

\begin{prop}\label{Prop: Unique continuation}
	Let $\varnothing \neq A \subset G$ be a closed subset and $f \in H^1(G)$. Then there exists a unique function $\widetilde{f} \in H^1(G)$ such that
	\begin{equation} \label{eq: harm extension clsd set}
		\begin{cases}
			\widetilde{f}|_{A} = f|_{A}, \\
			\widetilde{f} \text{ is harmonic on } G \setminus A.
		\end{cases}
	\end{equation}
\end{prop}

\begin{proof}
	Consider the connected components $O_i$, $i \in I$, of $G \setminus A$. Since the topological boundary $\partial O_i$ is finite, applying Proposition~\ref{Prop: Harmonic extension} to $\overline O_i = O \cup \partial O_i$, we find a unique function $f_i \in H^1(\overline O_i)$ with $f_i = f$ on $\partial O_i$. Glueing together the functions $f_i$ on $O_i$, $i \in  I$, with $f$ on $A$, we arrive at a unique function $\widetilde{f}$ with the above properties.
\end{proof}

When considering harmonic functions, we will often simplify complicated graphs using the following proposition, which corresponds to parallel series in physics.

\begin{defn}[Pumpkins] \label{def: pumpkins}
	Let $n\in \N$. An \emph{$n$-pumpkin} is a metric graph consisting of two vertices joined by $n$ distinct edges.

	Let $m\in \N$ and $n_1, \dots, n_m \in \N$. A \emph{$[n_1,\dots, n_m]$-pumpkin chain} is a metric graph with $m + 1$ vertices $x_0, x_1, \dots, x_m$ and $n_1 + n_2 + \dots + n_m$ edges, out of which $n_i$ join $x_{i-1}$ and $x_{i}$ for all $i=1, \dots, m$.
\end{defn}

\begin{prop}[Parallel series for pumpkin subgraphs] \label{prop: surgery}
	Let $G=(V,E, \ell)$ be a metric graph and $f \in H^1(G)$ harmonic on an open subset $O\subset G$.
	Suppose $ \overline{O}$ contains two vertices $x_1, x_2$ and $n$ edges $e_1, \dots, e_n \subset O$ between $x_1 $and $x_2$.

	Let $G'$ be the graph obtained by replacing the $n$ edges by a new edge $e$ connecting $x_1$ and $x_2$ with length $\ell'(e)=\left(\sum_{i=1}^{n} \ell(e_i)^{-1}\right)^{-1}$. Denote by $\Phi\colon G \to G'$ the natural surjective map, obtained by linearly dilating every edge $e_i$, $i = 1, \dots, n$, onto $e$.

	Then there exists a unique function $g$ on $G'$ with $g \circ \Phi = f$. Moreover, $g$ is harmonic on $\Phi(O)$ and $q(g) = q(f)$.
\end{prop}
\begin{proof}
	By construction, we can identify $G' \setminus e$ with $G \setminus (e_1 \cup \dots \cup e_n)$. Let $g \in H^1(G')$ be the unique function on $G'$ such that $g = f$ on $G' \setminus e$ and $g$ is linear on $e$. Since $f$ is linear on each edge $e_i$ by Lemma~\ref{Prop: characterization harmonicity}, we obtain that $g \circ \Phi = f$, which proves existence of $g$. Uniqueness follows by a similar argument.

	The harmonicity of $g$ on $\Phi(O)$ follows from Lemma~\ref{Prop: characterization harmonicity} and the fact that
	\[
		g_e'(x_1)= \frac{f(x_2) - f(x_1)}{\ell'(e)} = \sum_{i=1}^n \frac{f(x_2) - f(x_1)}{\ell(e_i)} = \sum_{i=1}^n f_{e_i}'(x_1).
	\]
	To compare the energies, note first that
	\begin{align*}
		q(g|_e) = \frac{|f(x_1) - f(x_2)|^2}{\ell'(e)} & = |f(x_1) - f(x_2)|^2 \sum_{i=1}^{n} \frac{1}{\ell(e_i)} = q(f|_{e_1\cup \dots \cup e_n}).
	\end{align*}\\
	Since $g = f$ on $  G' \setminus e   \cong G \setminus (e_1 \cup \dots \cup e_n) $, we conclude that $q(g) = q(f)$.
\end{proof}

\section{Laplacians on metric graphs with measures}	 \label{sec:LaplaciansMeasures}
In the following, we discuss how to associate a Laplacian to a graph equipped with a measure. Throughout this section, we let $G = (V,E, \ell)$ be a metric graph.
\subsection{Measures and Sobolev spaces on $G$}
Let
\[
	\cM:= \left\{\mu \text{ Borel measure on } G \; | \; \mu(G)=\int_{G}d\mu<\infty \right\}
\]
be the set of all Borel measures on $G$ having finite mass and
\[
	\cN := \left\{\mu\in \cM \; | \; \mu(G)=1 \right\} \subset \cM
\]
the subset consisting of all probability measures. The support of a measure $\mu \in  \cM$ is denoted by $\supp \mu$.

Note that we can identify $\mu  \in \cM$ with the positive linear functional $\varphi_\mu\colon C(G) \to \R$, $f \mapsto \int_G f d\mu$. This allows to view $\cM$ as a subset of the dual $(C(G)^\ast, \|\cdot\|_{C(G)^\ast})$ of $(C(G), \|\cdot\|_{\infty})$. In the following, we endow $\cM$ with the weak-$\ast$ topology. By the Riesz–Markov–Kakutani representation theorem, we have $\|\mu\|_{C(G)^\ast} = \mu(G)$ for all $\mu \in \cM$. In particular, by the Banach–Alaoglu theorem, $\cN$ with the above topology is compact.

A first observation is the following inequality.
\begin{lem}\label{Prop: Nicaise}
	Let $G$ be a metric graph and $\mu \in \cM$. Suppose $f\in H^1(G)$ has at least one zero. Then
	\begin{equation}\label{eq: Dirichlet Nicaise}
		q(f) \diam(G) \mu(G)\geq \|f\|_\infty^2 \mu(G)\geq \int_G |f|^2 d\mu,
	\end{equation}
	where $\diam(G)= \diam_\varrho(G)$ is the diameter of $G$ with respect to the path metric $\varrho$.
\end{lem}
\begin{proof}
	The second inequality is obvious. To prove the first inequality, fix $x_0\in G$ with $|f(x_0)| = \|f\|_\infty$. Let $\gamma \subset G$ be a shortest path connecting $x_0$ with an arbitrary zero of $f$. Since $\gamma$ has length smaller than $\diam(G)$, we obtain
	\[
		\|f\|_\infty^2= |f(x_0)|^2 = \left|\int_\gamma f'(s)ds\right|^2 \leq \diam(G) \int_{G} |f'|^2 dx  =  \diam(G) q(f). \qedhere
	\]
\end{proof}

For $\mu \in \cM$, we let $L^2(G, \mu)$ be the corresponding $L^2$-space consisting of (equivalence classes of) measurable functions $f \colon G \to \R$ such that
\[
	\|f\|^2_{L^2(G,\mu)} := \int_{G}|f|^2d\mu < \infty.
\]
Furthermore, on the Sobolev space $H^1(G)$ we define the scalar product
\[
	(f,g)_{H^1, \mu} := \int_G fg \, d\mu + \int_{G} f'g' d x =(f,g)_{L^2(G, \mu)} + q(f,g).
\]
The corresponding norm is denoted by $\|\cdot \|_{H^1, \mu}$.

The next lemma ensures that these different norms on $H^1(G)$ are all equivalent:

\begin{lem} \label{lem:EquivalenceH^1Norms} Consider a metric graph $G$.
	\begin{itemize}
		\item [(i)] Let $\mu_1, \mu_2 \in \cM$. Then
		      \begin{equation} \label{eq:EquivalenceH^1NormsAuxiliary}
			      \frac{1}{\sqrt{\mu_1(G)}} \|f\|_{L^2(G, \mu_1)} \le \frac{1}{\sqrt{\mu_2(G)}}  \|f\|_{L^2(G,\mu_2)} + \sqrt{\diam (G) q(f)}
		      \end{equation}
		      for all $f \in H^1(G)$.

		\item [(ii)] There exists a uniform constant $C >0$ such that for all $\mu \in \cN$ the estimate
		      \[
			      \frac{1}{C} \|f \|_{H^1, \mu} \le \|f\|_{H^1} \le C \|f\|_{H^1, \mu}, \qquad f \in H^1(G),
		      \]
		      holds.
		\item [(iii)] $(H^1(G), (\cdot, \cdot)_{H^1, \mu})$ is a Hilbert space for all $\mu \in \cM$.
	\end{itemize}
\end{lem}
\begin{proof} Upon replacing measures $\mu\in\cM$ by their normalizations $\widetilde \mu = \mu/\mu(G)$, it suffices to prove the estimate~\eqref{eq:EquivalenceH^1NormsAuxiliary} for $\mu_1, \mu_2 \in \cN$. Then every function $f\in H^1(G)$ can be decomposed as $f=c\one_G + f_0$ with
	\[
		c= \int_G f \, d \mu_2 \text{ and } f_0 \in H^1(G) \text{ such that } \int_{G} f_0 \, d \mu_2 = 0.
	\]
	Since $f_0$ must have at least one zero, $\|f_0\|_\infty \leq \sqrt{q(f) \diam(G)}$ by Lemma \ref{Prop: Nicaise}. Then
	\begin{align*}\label{L2muH1ineq}
		\|f\|_{L^2(G,\mu_1)} & \leq \|f\|_{\infty}\leq |c|+ \|f_0\|_{\infty} \leq \int_{G} |f| \, d\mu_2 + \sqrt{q(f)\diam(G)} \\
		                     & \leq \|f\|_{L^2(G, \mu_2)} + \sqrt{\diam(G)q(f)}.
	\end{align*}
	Choosing $\mu_2$ as the Lebesgue measure $dx$, one readily obtains the two-sided estimate in (ii). Finally, since $(H^1(G), (\cdot, \cdot)_{H^1})$ is a Hilbert space and  $\|\cdot\|_{H^1}$ and $\|\cdot\|_{H^1, \mu}$ are equivalent, the same holds for $(H^1(G), (\cdot, \cdot)_{H^1, \mu})$.
\end{proof}

In the following, we want to consider $H^1(G)$ as a subspace of $L^2(G, \mu)$. If $\supp\mu \neq G$, two functions $f,g \in H^1(G)$ can differ on $G\setminus \supp \mu$ while the equality $f =g$ still holds in $L^2(G, \mu)$. To remedy this, we define the closed subspace
\[
	H^1(G, \mu) := \{f \in H^1(G) \; | \;  f \text{ is harmonic on } G\setminus \supp \mu  \} \subset H^1(G).
\]
Equipped with $(\cdot, \cdot)_{H^1, \mu}$, this space becomes a Hilbert space. For the special case $\mu = dx$, we obviously have $H^1(G, dx) = H^1(G)$ as Hilbert spaces. If $\supp \mu = G$, then $H^1(G, \mu ) = H^1(G)$ as vector spaces.

\begin{defn}
	Let $\mu \in \cM$. The orthogonal projection from $(H^1(G), (\cdot, \cdot)_{H^1, \mu})$ onto the closed subspace $H^1(G, \mu) \subset H^1(G)$ is denoted by
	\[
		\cP_\mu : H^1(G) \to H^1(G, \mu).
	\]
\end{defn}
Note that the image $\cP_\mu(f)$ of a function $f\in H^1(G)$ is precisely the unique solution of the problem \eqref{eq: harm extension clsd set} for $A=\supp \mu$.

\begin{lem}\label{lem:OrthoDecompositionH^1}
	Let $H^1_0(G \setminus \supp \mu) := \{f \in H^1(G) \; | \; f\big|_{ \supp \mu} \equiv 0  \}$. Then
	\[
		H^1(G) = H^1(G, \mu) \oplus H^1_0(G \setminus \supp \mu)
	\]
	is an orthogonal decomposition regarding the scalar product $(\cdot, \cdot)_{H^1, \mu}$. Moreover, the decomposition is orthogonal with respect to $q(\cdot)$ and $(\cdot, \cdot)_{L^2(G,\mu)}$, meaning that
	\[
		q(f,g) = \int_G f g \, d\mu = 0
	\]
	for all $f \in H^1(G, \mu)$ and $g \in H^1_0(G \setminus \supp \mu)$.
\end{lem}
\begin{proof}
	The orthogonality in the $L^2(G, \mu)$--scalar product is a direct consequence of the definition of $H^1_0(G \setminus \supp \mu)$. The orthogonality regarding the energy form follows from the definition of harmonicity on $G \setminus \supp \mu$.
\end{proof}

Note that the natural map $f \mapsto [f]$, sending every continuous function $f$ to the corresponding equivalence class in $L^2(G, \mu)$, is injective from $H^1(G, \mu)$ to $L^2(G, \mu)$. In particular, we have a natural inclusion $H^1(G, \mu)  \subset L^2(G, \mu)$.

\begin{lem} \label{compembed}
	Let $\mu \in \cM$. Then the embeddings
	\[
		(H^1(G, \mu), \|\cdot\|_{H^1, \mu}) \hookrightarrow (C(G), \|\cdot\|_\infty)
	\]
	and
	\[
		(H^1(G, \mu), \|\cdot\|_{H^1, \mu}) \hookrightarrow (L^2(G,  \mu), \|\cdot\|_{L^2(G, \mu)})
	\]
	are compact.
\end{lem}
\begin{proof}
	The compactness of the embedding $(H^1(G), \|\cdot\|_{H^1}) \hookrightarrow (C(G), \|\cdot\|_\infty)$ is well-known (see, e.g., \cite[Theorem 8.8]{brezis} for the statement for intervals). Since the norms $\|\cdot\|_{H^1, \mu}$ and $\|\cdot\|_{H^1}$ are equivalent by Lemma~\ref{lem:EquivalenceH^1Norms}, the first claim follows.\\
	Now consider the map $\iota$ sending each function $f \in (C(G), \|\cdot\|_{\infty})$ to its corresponding equivalence class $[f] \in (L^2(G, \mu), \|\cdot\|_{L^2(G, \mu)})$. Clearly $\iota$ is bounded and thus composing this map with the compact embedding $(H^1(G, \mu), \|\cdot\|_{H^1, \mu}) \hookrightarrow (C(G), \|\cdot\|_\infty)$, we obtain a compact operator. Since this operator is also injective (due to the definition of $H^1(G, \mu)$), it is in fact a compact embedding.
\end{proof}

\subsection{Laplacians on metric graphs with measures} \label{ss:LaplacianBasics}

Viewing $H^1(G, \mu)$ as a subspace of $L^2(G,\mu)$, we consider $q\big|_{H^1(G, \mu)}$ as a quadratic form in $L^2(G,\mu)$.

\begin{lem} The quadratic form $q\big|_{H^1(G, \mu)}$ is a densely defined, closed and non-negative form in $L^2(G,\mu)$.
\end{lem}
\begin{proof}
	It is clear that $q\geq 0$. To show that $H^1(G, \mu)$ is dense in $L^2(G, \mu)$ it suffices to show the statement only on $\supp \mu$, i.e., we may assume that $\supp \mu =G$. But $C(G) \subset L^2(G)$ is dense in $\|\cdot\|_{L^2(G,\mu)}$-norm by Lusin's theorem, and $H^1(G) \subset C(G)$ is dense in $\|\cdot\|_{\infty}$-norm by the Stone-Weierstrass theorem, so the claim holds. Moreover, $q\big|_{H^1(G, \mu)}$ is closed by Lemma~\ref{lem:EquivalenceH^1Norms} and the fact that $H^1(G, \mu) \subset H^1(G)$ is a closed subspace.
\end{proof}
Consequently, $q$ generates a non-negative, self-adjoint operator
\begin{align*}
	H_\mu\colon \dom(H_\mu) \subset L^2(G, \mu) \longrightarrow L^2(G, \mu)
\end{align*}
defined by
\begin{gather*}
	f \in \dom(H_\mu) \text{ and } H_\mu f = g \, \Leftrightarrow \\ f \in H^1(G, \mu) \text{ and } q(f,\varphi) = \int_G g \varphi \, d\mu \text{ for all } \varphi \in H^1(G, \mu).
\end{gather*}
Using the orthogonal decomposition in Lemma~\ref{lem:OrthoDecompositionH^1}, one can prove that the equality
\begin{equation} \label{eq:RepresentationTheorem}
	q(f,\varphi) = \int_G H_\mu f \varphi \, d\mu
\end{equation}
holds more strongly for all $f \in\dom(H_\mu)$ and $\varphi \in H^1(G)$.

\begin{lem}\label{lem:BasicsSpecHmu}
	Let  $\mu\in \cM$. Then $H_\mu$ has discrete spectrum and $\ker(H_\mu) =\Span(\mathbbm{1}_G)$.
\end{lem}
\begin{proof}
	Discreteness of the spectrum is equivalent to the compactness of the embedding $H^1(G, \mu) \to L^2(G, \mu)$, which was proved in Lemma~\ref{compembed}. We have $f \in \ker H_\mu$ exactly when $f'\equiv 0$. Since $G$ is connected, this holds exactly for constant $f$.
\end{proof}

The class of Laplacians $H_\mu$, $\mu \in \cM$, contains the following examples.

\begin{ex}[Weighted discrete Laplacians]  \label{ex:DiscreteLaplacians}
	Let $\mu \in\cM$ be a measure of the form
	\[
		\mu = \sum_{v\in V} m(v) \delta_v
	\]
	where $\delta_v$ denotes the Dirac measure at $v \in V$ and $m\colon V\to (0, \infty)$ is a weight function. Since $\supp \mu  = V$, the corresponding $L^2$-space is finite dimensional and coincides with the space $\ell^2(V,m) =\{f \colon V\to \R \}$  equipped with the norm
	\[
		\|f\|_m = \Big ( \sum_{v \in V} |f(v)|^2 m(v)\Big )^{1/2}, \qquad \text{for } f \colon V\to \R.
	\]
	Under this identification, the operator $H_\mu$ has domain $\dom(H_\mu) =\ell^2(V, m)$ and acts on $f \in \ell^2(V,m)$ as a {\em weighted discrete Laplacian} given by

	\begin{equation} \label{eq:weighted discrete Lap}
		(H_\mu f)(v)  =  \frac{1}{m(v)} \sum_{w \sim v}\sum_{e = vw} \frac{1}{\ell(e)} (f(v) - f(w)),  \qquad v \in V.
	\end{equation}
\end{ex}

\begin{ex}[Weighted Dirichlet-to-Neumann operators on graphs] \label{ex:DtN} For a subset $B  \subset  V$ and a function $m\colon  B\to (0, \infty)$, consider the measure
	\[
		\mu = \sum_{b \in B} m(b) \delta_b.
	\]
	The $L^2$-space $L^2(G,\mu)$ corresponds to $\ell^2(B,m) =\{f \colon B\to \R \}$ with the norm
	\[
		\|f\|_m = \Big ( \sum_{b \in B} |f(b)|^2 m(b)\Big )^{1/2}, \qquad \text{for } f \colon B\to \R.
	\]
	The operator $H_\mu$ has domain $\dom(H_\mu) =\ell^2(B, m)$. By Proposition~\ref{Prop: Harmonic extension}, every function $f \in \ell^2(B,m)$ has a unique extension $\tilde f \in H^1(G)$ such that $f=\tilde f$ on $B$ and $\tilde f$ is harmonic on $G \setminus B$. In fact, $\tilde f$ is linear on every edge $e$ and satisfies $ \sum_{w \sim v}\sum_{e  = vw } \frac{1}{\ell(e)} (\tilde f (v) - \tilde f (w)) =0$ at every vertex $v \notin B$.  As follows from Proposition~\ref{Prop: Harmonic extension}, the image $H_\mu f$ is then given by
	\[
		(H_\mu f)(v) =   \frac{1}{m(v)} \sum_{w \sim v}\sum_{e \in vw} \frac{1}{\ell(e)} (\tilde f (v) - \tilde f (w)),  \qquad v \in B.
	\]
	Operators of this type are called {\em Dirichlet-to-Neumann operators on graphs} in the literature. Their eigenvalues are often referred to as {\em Steklov eigenvalues}. If $B = V$, then $H_\mu$ is a weighted discrete Laplacian from Example~\ref{ex:DiscreteLaplacians}.
\end{ex}

\begin{rmk} \label{rem:DtNGeneral}
	In Example~\ref{ex:DtN}, we can actually consider a general finite subset $B \subset G$. Namely, the condition $B \subset V$ always holds after choosing a model which contains $B$ in its vertex set $V$ (which can be done by suitably subdividing edges).
\end{rmk}

\begin{ex}[Weighted Kirchhoff Laplacians]\label{ex:WeightedKirchhoffLaplacian}
	Fix a strictly positive function $m  \in L^1(G)$. Consider the measure $\mu = m(x) \,dx$ and the corresponding Laplacian $H_\mu \colon \dom(H_\mu) \subset L^2(G, \mu) \to  L^2(G, \mu)$. Introduce also the space
	\[
		H^2(G\setminus V, \mu) =  \bigoplus_{e \in E} \Big \{f_e \in  \mathrm{AC}([0, \ell(e)]) \colon   f_e' \in AC[0, \ell(e)], \, \frac{f_e''}{m} \in L^2((0, \ell(e)), m \,dx_e) \Big \}
	\]
	where $\AC[0, \ell(e)]$ denotes the space of absolutely continuous functions and $dx_e$ the Lebesgue measure on an interval edge $[0, \ell(e)]$, $e \in E$. Then the operator $H_\mu$ acts edgewise as the weighted second derivative
	\[
		H_\mu f (x) =  -   \frac{1}{m(x)} f''(x)
	\]
	on the domain
	\[
		\dom(H_\mu) = \{f =(f_e)_{e \in E} \in H^2(G \setminus V, \mu)| \, \text{ $f$ satisfies \eqref{eq:kirchhoff} at every vertex $v \in V$} \},
	\]
	where the so-called {\em Kirchhoff conditions} at a vertex $v  \in V$ are given by
	\begin{align}\label{eq:kirchhoff}
		\begin{cases} f\ \text{is continuous at}\ v, \\[1mm]
			\sum\limits_{e \sim v} f_e'(v) = 0,\end{cases}
	\end{align}
	and we sum over all edges $e$ incident to $v$, which are all oriented towards $v$.

	For $\mu = d x$, the operator $H_\mu =H_{dx}$ is the standard Kirchhoff Laplacian on $G$.
\end{ex}

\begin{ex}[Limits of shrinking manifolds] Consider a measure on $G$ of the form
	\[
		\mu = \sum_{e \in E} \mu_e dx_e + \sum_{v \in V} \mu_v \delta_v
	\]
	for coefficients $(\mu_e)_{e \in E}, (\mu_v)_{v \in V}\subset (0, \infty)$.

	The corresponding Laplacian $H_\mu \colon \dom(H_\mu) \subset L^2(G, \mu)  \to L^2(G,\mu)$ has domain
	\[
		\dom(H_\mu) =  \{f\in L^2(G, \mu)| \, \text{$f$ is continuous and } f_e \in H^2((0, \ell(e)) \text{ for all edges $e$} \},
	\]
	where $H^2((a,b))$ is the second order Sobolev space on an interval $(a,b)$.

	Moreover, $H_\mu$ maps $f \in \dom (H_\mu)$ to the function $H_\mu f \in L^2(G, \mu)$ given by
	\[
		(H_\mu f)(x) = \begin{cases}
			- \frac{1}{\mu_e} f''(x),                 & \text{if $x$ lies in the interior of $e \in E$} \\
			\frac{1}{\mu_x} \sum_{e \sim x}  f_e'(x), & \text{if $x$ is a vertex}
		\end{cases}
	\]
	where we sum over the edges $e$ incident to $x$, which are all oriented towards $x$.

	Such Laplacians appear as limits of Laplacians on shrinking manifolds, if different parts of the manifolds shrink at a specific relative speed, see~\cite[Section 7]{ep05} and \cite[Section 7.2.3]{post}.

\end{ex}


\section{Optimal eigenvalues on a metric graph} \label{sec:DefinitionEigenvalues}

Throughout this section, let $G = (V,E, \ell)$ be a metric graph. Consider a measure $\mu \in \cM$ and let $H_\mu$ be the associated Laplacian acting in $L^2(G, \mu)$. By Lemma~\ref{lem:BasicsSpecHmu}, the spectrum of $H_\mu$ consists of an increasing sequence of eigenvalues of finite multiplicity, which we enumerate, counting multiplicities, as
\[
	\lambda_0(H_\mu) < \lambda_1(H_\mu) \le \lambda_2(H_\mu) \le \dots.
\]
If $\dim(L^2(G, \mu)) < \infty$ or, equivalently, the support of $\mu$ is finite, then this sequence is finite; otherwise, the sequence is infinite. For convenience, we set $\lambda_k(H_\mu) := \infty$ if $k \geq \dim(L^2(G, \mu))$ and index the sequence over $k$ in $\N_0 := \N\cup \{0\}$ in both cases.

We have already seen that $\lambda_0(H_\mu) = 0$ for all $\mu  \in \cM$ (Lemma \ref{lem:BasicsSpecHmu}). For $k \ge 1$, we have the variational characterizations
\[
	\lambda_k(H_\mu) = \min\limits_{\substack{F \subset H^1(G,\mu) \\ \dim F = k+1}} \max\limits_{f \in F\setminus\{0\} } \;R_\mu(f),
\]
where $R_\mu(f) := {q(f)}/{\|f\|^2_{L^2(G;\mu)}}$ is the Rayleigh quotient of $f \in H^1(G, \mu)$, and
\[
	\lambda_k(H_\mu) =\max\limits_{\substack{f_0, \dots, f_{k-1} \in H^1(G,\mu)}} \min\limits_{\substack{f \in H^1(G, \mu)\setminus\{0\} \\ f \perp_\mu f_0, \dots, f_{k-1}}} R_\mu (f).
\]
In particular, the first eigenvalue $\lambda_1(H_\mu)$ equals
\begin{equation} \label{eq:MinLambda1}
	\lambda_1(H_\mu) =\min_{\substack{f \in H^1(G,\mu) \\ f \neq  0,\; f \perp_{\mu} \one}} R_\mu (f) =  \min_{\substack{f \in H^1(G)\\ f \neq  0,\; f \perp_\mu \one}}R_\mu (f),
\end{equation}
where the second equality follows from Lemma~\ref{lem:OrthoDecompositionH^1}.

In the following, we are interested in maximizing and minimizing the $k$-th eigenvalue $\lambda_k(H_\mu)$ with respect to the measure $\mu$. The following continuity result provides the basis of our study.

\begin{thm}[Spectral continuity]\label{thm: spec conv }
	Let $(\mu_m)_{m \in\N} \subset \cM$ be a sequence of measures converging to a measure $\mu \in \cM$ in the weak-$\ast$ topology. Fix $k \in \N_{0}$. Then:
	\begin{enumerate}[label={(\roman*)},ref={(\roman*)}]
		\item \label{thm:SpectralConvergence1} $\lim_{m \to  \infty} \lambda_k(H_{\mu_m}) =  \lambda_k(H_\mu)$ in $[0, \infty]$.
		\item Assume that $\lambda_k(H_\mu) < \infty$, so that $\lambda_k(H_{\mu_m}) < \infty$ for all large $m$ by (i). Let $v^0_m, \dots, v^k_m$ be orthonormal eigenfunctions of $H_{\mu_m}$ for the eigenvalues $\lambda_0(H_{\mu_m}), \dots, \lambda_k(H_{\mu_m})$. Then, upon passing to a subsequence, $v^0_m, \dots, v^k_m$ converge uniformly on $G$ to orthonormal eigenfunctions $v^0, \dots, v^k$ of $H_\mu$ for the eigenvalues $\lambda_0(H_{\mu}), \dots, \lambda_k(H_{\mu})$.
	\end{enumerate}

\end{thm}

We stress that the convergence in Theorem~\ref{thm: spec conv } (i) holds in $[0, \infty]$, that is, we allow that $\lambda_k(H_\mu) =+ \infty$ or $\lambda_k(H_{\mu_m}) = + \infty$.

\begin{proof}
	We prove the statement by induction on $k$. The base case $k=0$ is trivial by Lemma~\ref{lem:BasicsSpecHmu}. Suppose the statement is true for $k-1$. We first assume that $\lambda_k(H_\mu) < \infty$. Then the support of $\mu$ must contain at least $(k+1)$ points. By the weak-$\ast$ convergence, the same must hold for the measure $\mu_m$ if $m$ is large. In particular, we must have $\lambda_k(H_{\mu_m})  < \infty$ for all large $m$.

	Fix $\mu_m$-orthonormal eigenfunctions $v^{0}_m,\dots,v^{k-1}_m$ of $H_{\mu_m}$ for the eigenvalues $\lambda_0(H_{\mu_m}), \dots, \lambda_{k-1}(H_{\mu_m})$. By the induction hypothesis, passing to a subsequence, we may assume that the functions $v^j_m$ converge uniformly on $G$ to $\mu$-orthonormal eigenfunctions $v^j$ of $H_\mu$ for the eigenvalues $\lambda_0(H_\mu), \dots \lambda_{k-1}(H_\mu)$.

	We first show that
	\begin{equation} \label{eq:ContinuityLimsup k}
		\limsup_{m \to \infty} \lambda_k(H_{\mu_m}) \le \lambda_k(H_\mu).
	\end{equation}
	Choose a function $v \in H^1(G, \mu)$ with $\lambda_k(H_\mu) = R_\mu(v)$ and $v\perp_{\mu} v^j$ for all $j<k$. For $m\in \N$, the function
	\[
		v_m := \cP_{\mu_m} (v) - \sum_{j=0}^{k-1} \left( \int_{G} v\cdot v_{m}^j d\mu_m\right) v_{m}^j
	\]
	belongs to $H^1(G, \mu_m)$ and satisfies $v_m \perp_{\mu_m} v_{m}^j $ for all $j< k$. In addition,
	\[
		q(v_{m}^j, v_{m}^i)  = \lambda_i(H_{\mu_m}) \delta_{i,j}
	\]
	and thus, using~\eqref{eq:RepresentationTheorem},
	\[
		q(v_m) = q(\cP_{\mu_{m}}(v)) - \sum_{j=0}^{k-1} \left(\int_{G}v \cdot v_{m}^j d\mu_{m}\right)^2\lambda_j(H_{\mu_{m}}) \le q(\cP_{\mu_{m}}(v)).
	\]
	Since $\mu_m$ converges to $\mu$ in the weak-$\ast$ topology and $v_{m}^j$ converges to $v^j$ uniformly, we have, as $m \to \infty$,
	\[
		\int_{G}v \cdot v_{m}^j d \mu_m \to \int_{G} v\cdot v^j d\mu = 0 \qquad \text{ for all } j < k.
	\]
	Moreover, with $\cP_{\mu_m}=\operatorname{id}$ in $L^2(G, \mu_m)$ and $q(\cP_{\mu_m} (\cdot)) \leq q(\cdot)$, we get
	\[
		\int_G v_m^2 \, d \mu_m \to  \int_G v^2 \, d\mu, \qquad   R_{\mu_m}(v_m) \leq \frac{q(v)}{\int_G v_m^2 \, d \mu_m} \to R_\mu(v)
	\]
	as $m \to \infty$. In particular,
	\[
		\limsup_{m \to \infty} \lambda_k(H_{\mu_m}) \le \limsup_{m \to \infty} R_{\mu_m}(v_m) \leq R_\mu(v) = \lambda_k(H_\mu)
	\]
	and \eqref{eq:ContinuityLimsup k} is proved.

	Next we prove that, up to restricting to a suitable subsequence, we also have
	\begin{equation} \label{eq:ContinuityLimin k}
		\lambda_k(H_\mu) \le \liminf_{m \to \infty} \lambda_k(H_{\mu_m}).
	\end{equation}
	Since a sequence $(a_n)_n \subset \R$ converges to $a \in \R$ if and only if every subsequence of $(a_n)_n$ has a subsequence on which the convergence holds, applying the above to arbitrary subsequences of $\mu_m$, this then implies that $\lim_{m \to \infty} \lambda_k(H_{\mu_m}) = \lambda_k(H_{\mu})$.

	Fix $\varepsilon>0$. Then \eqref{eq:ContinuityLimsup k} implies that
	\[
		q(v_{m}^k)=R_{\mu_m}(v_{m}^k) = \lambda_k(H_{\mu_{m}}) \le \lambda_k(H_\mu) + \varepsilon
	\]
	for all large $m$. In particular, the sequence $(v_{m}^k)_m$ is bounded in $(H^1(G), \|\cdot\|_{H^1})$ by Lemma~\ref{lem:EquivalenceH^1Norms}. Applying Lemma~\ref{compembed}, up to restricting to a subsequence, we can assume that $(v_m^k)_m$ converges uniformly to a function $v^k \in C(G)$.
	Using the definition of weak-$\ast$ topology, one readily obtains that
	\[
		1 = \lim_{m \to \infty} \int_G \left(v_{m}^k\right)^2 \, d\mu_m = \int_G \left(v^k\right)^2 \, d\mu
	\]
	and that, since $v_{m}^j$ converges uniformly to $v^j$ for all $j<k$,
	\[
		\qquad 0 = \lim_{m \to \infty} \int_G v_{m}^kv_{m}^j \, d\mu_m = \int_Gv^k v^j \, d\mu.
	\]
	We also have that $(v_m^k)_m$ converges to $v^k$ in $L^2(G) =L^2(G,dx)$. Moreover, since the quadratic form $q$ is closed in $L^2(G)$, it is lower semi-continuous in $L^2(G)$ (see, e.g., \cite[Theorem 1.2.1]{dav89}). Thus, we conclude that $v^k$ belongs to the form domain $\dom(q) = H^1(G)$ and
	\[
		q(v^k) \le \liminf_{m \to \infty} q(v_{m}^k).
	\]
	Combining the above statement, we obtain that
	\begin{align} \begin{split}
			\label{eq: proof continuity of eigenvalues}
			\liminf_{m \to \infty} \lambda_k(H_{\mu_m}) =  \liminf_{m \to \infty} \frac{q(v_{m}^k)}{\int_G \left(v_{m}^k\right)^2 \, d\mu_m}
			\ge \frac{q(v^k)}{\int_G \left(v^k\right)^2 \, d\mu} = R_\mu(v^k).
		\end{split}
	\end{align}
	Since $(v^0,\dots, v^k)$ are mutually orthogonal, by the min-max-principle it follows that $R_\mu(v^k) \geq R_\mu( \cP_\mu(v^k))  \geq \lambda_k(H_\mu)$ and hence \eqref{eq:ContinuityLimin k} holds. This finishes the proof of the claim in (i).

	The mutual orthogonality of $(v^0, \dots, v^k)$ as well as $\|v^k\|_{L^2(G,\mu)} = 1$ have already been proved. It remains to show that $v^k \in H^1(G,\mu)$ is an eigenfunction for $\lambda_k(H_\mu)$. Using the above and taking into account that $\cP_\mu v^k$ is orthogonal to $v^0, \dots, v^{k-1}$ in $L^2(G, \mu)$, we already have
	\[
		R_\mu(v^k) \leq \liminf_{m \to \infty} \lambda_k(H_{\mu_m}) = \lambda_k(H_\mu) \leq \frac{q(\cP_\mu v^k)}{\int_{G} (\cP_\mu v^k)^2 d\mu} \leq \frac{q(v^k)}{\int_{G} \left(v^k\right)^2 d\mu} = R_\mu (v^k).
	\]
	Thus $R_\mu(v^k) = \lambda_k(H_\mu)$. Further, the last inequality shows that $q(\cP_\mu v^k) = q(v^k)$. Since $\cP_\mu v^k = v^k$ on $\supp \mu$, we get $\cP_\mu v^k \equiv v^k$ and $v^k \in H^1(G, \mu)$ by Proposition~\ref{Prop: Harmonic extension}. Thus $H_\mu v^k = \lambda_k(H_\mu) v^k$, and $v^k$ is in fact a eigenfunction for $\lambda_k(H_\mu)$.

	To complete the proof, we treat the case that $\lambda_k(H_\mu) = + \infty$ or, equivalently, $\dim L^2(G,\mu) \leq k$. Assume there exists a sequence $(\mu_m)_m \subset \cM$ converging to $\mu$ with $\liminf_{m \to \infty} \lambda_k(H_{\mu_m}) < + \infty$. By the above argument, for a subsequence, the eigenfunctions $v_m^0, \dots, v_{m}^k$ converge uniformly to mutually orthogonal functions $v^0, \dots, v^k \in  L^2(G,\mu) $. But this contradicts the fact that $ \dim L^2(G,\mu)  \leq k$.
\end{proof}

In order to prevent scaling problems, we restrict to our optimization problems to measures in $\cN \subset \cM$. It turns out that the question of maximizing the $k$-th eigenvalue has a trivial answer:
\begin{cor} \label{cor:MaximumMeaningless}
	Let $\widetilde \cN$ be a dense subset of $\cN$. Then
	\[
		\sup_{\mu \in \widetilde \cN} \lambda_k(H_\mu) = + \infty \qquad \text{for all $k \ge 1$}.
	\]
\end{cor}
\begin{proof}
	By density, $\widetilde \cN$ contains a sequence $(\mu_m)_m$ converging to a Dirac measure $\mu = \delta_x$. But then $\lambda_k(H_{\mu_m}) \to \lambda_k(H_\mu) = + \infty$ for $m \to \infty$ by Theorem~\ref{thm: spec conv }.
\end{proof}
Note that in Corollary~\ref{cor:MaximumMeaningless}, we can choose $\widetilde \cN$ as the set of probability measures with finite support, or having the form $\mu = f dx$ for a smooth function $f >0$.

On the other hand, the minimum of $\lambda_k(H_\mu)$ over $\cN$ exists by Theorem~\ref{thm: spec conv }.
\begin{defn}
	Let $G$ be a metric graph and $k \in \N_{0}$. The \emph{$k$-th optimal eigenvalue} of $G$ is defined by
	\[
		\lambda_k^{\min}(G) := \min_{\mu \in \cN } \lambda_k(H_\mu).
	\]
\end{defn}
As is clear from Lemma~\ref{lem:BasicsSpecHmu}, we have $\lambda_0^{\min}(G) = 0$ and $\lambda_k^{\min}(G) >0$ for all $k \ge 1$.

\begin{rmk} \label{rmk : Scaling property}
	Multiplying all edge lengths in $G$ by $c >0$, we obtain another metric graph $G_c$ of total length $L(G_c) = cL(G)$. It is not hard to see that
	\begin{equation}
		\lambda_k^{\min} (G) = c\lambda_k^{\min}(G_c).
	\end{equation}
	Thus, one can assume that $G$ has total length $L(G)=1$ when studying $ \lambda_k^{\min} (G)$.
\end{rmk}


\section{The first optimal eigenvalue} \label{sec:FirstEigenvalue}

In this section, we derive three different characterizations of the first optimal eigenvalue $\lambda_1^{\min}(G)$: in terms of two-point-measures, resistance metrics and graph partitions, respectively.

\subsection{Characterization via two-point-measures}
As it turns out, the first eigenvalue is minimized by a measure supported on two points. This reduces the analysis of the optimal eigenvalue to rather special measures.

\begin{thm} \label{thm: characerisation of lambda1}
	Let $G$ be a metric graph. Then there exists a measure $\mu$ of the form
	\begin{equation} \label{eq:FormMinimizingMeasure}
		\mu = \frac{1}{2}\delta_x + \frac{1}{2} \delta_y
	\end{equation}
	for two points $x\neq y \in G$ such that
	\[
		\lambda_1^{\min}(G) = \lambda_1(H_\mu).
	\]
	Moreover, every measure $\mu\in \cN$ with $\lambda_1^{\min}(G) = \lambda_1(H_\mu)$ is of the form~\eqref{eq:FormMinimizingMeasure} for some points $x\neq  y \in G$.
\end{thm}
Before proving Theorem~\ref{thm: characerisation of lambda1}, we introduce the following notation. For two points  $x \neq y \in G$, define the measure $\mu_{x,y} \in \cN$ as $\mu_{x,y} := \frac{1}{2}\delta_x + \frac{1}{2}\delta_y$ and the function $f_{x,y} \in H^1(G,\mu)$ as the unique solution of the equations
\begin{equation} \label{eq:Definitionfxy}
	\begin{cases}
		f(x)=1,    \\
		f(y) = -1, \\
		f \text{ harmonic}  \text{ on }G\setminus\{x,y\}.
	\end{cases}
\end{equation}
Then the function $f_{x,y}$ is the unique $\mu_{x,y}$-normalized eigenfunction (up to the sign) of the Laplacian $H_{\mu_{x,y}}$ for the eigenvalue $\lambda_1(H_{\mu_{x,y}})$. In terms of these notions, Theorem~\ref{thm: characerisation of lambda1} asserts that
\begin{equation} \label{eq:ExpressLamda1Fxy}
	\lambda_1^{\min} (G) = \min_{x \neq y \in G}  \lambda_1 (H_{\mu_{x,y}}) =\min_{x \neq y \in G} q(f_{x,y}).
\end{equation}

In order to prove Theorem~\ref{thm: characerisation of lambda1}, we first provide two lemmas.

\begin{lem} \label{lem:KeyLemma}
	Let $\mu \in \cN$ and let $f\colon G \to \R$ be a continuous function such that $\int_Gf \, d\mu =0$. Then:
	\begin{itemize}
		\item [(i)] $\|f\|_{L^2(G, \mu)}^2 \le |\max f| |\min f| $, where $\max f$ and $\min f$ denote the maximum and minimum of $f$, respectively.
		\item [(ii)]  There exist $x,y \in G$ and $\alpha \in (0,1)$ such that the measure $\nu = \alpha \delta_{x} + (1-\alpha) \delta_{y}$ satisfies
		      \[
			      \int_{G}|f|^2 \, d \nu = \int_G|f|^2 \, d\mu \qquad \text{and} \qquad \int_{G} f \, d\nu  = 0.
		      \]
	\end{itemize}
\end{lem}
\begin{proof}
	It suffices to prove the claims in case that $\|f\|_{L^2(G,\mu)} = 1$. Then
	\[
		\int_{G} (f- \min f)^2 d\mu = \int_{G} f^2 d \mu - 2 \int_{G} f \min f d \mu + \int_{G} (\min f)^2 d \mu = 1 + (\min f)^2.
	\]
	Since $f(x)-\min f \leq \max f - \min f$ for $ x \in G$, we further obtain the estimate
	\begin{align*}
		1 + (\min f)^2 & = \int_{G} (f- \min f )^2 d\mu \leq ( \max f  - \min f) \int_{G} f-\min f d\mu \\
		               & = ( \max f  - \min f ) (- \min f) = - \min f \max f + (\min f)^2,
	\end{align*}
	which proves the inequality in (i).

	In order to prove the claim in (ii), note first that $\max f  > 0 > \min f$ since $\int_G f \, d\mu = 0$. Combining the inequality in (i) with the continuity of $f$, we infer that there exist points $x,y  \in G$ with $f(x)f(y)=-1$. Setting $\alpha := \frac{1}{f(x)^2+1} < 1$, we get
	\[
		\int_{G} f d \nu = \alpha f(x) + (1-\alpha) f(y) = \frac{f(x)}{f(x)^2+1} + \frac{f(x)^2 f(y)}{f(x)^2+1} = 0
	\]
	and
	\[
		\int_{G} |f|^2 d \nu = \alpha f(x)^2 + (1- \alpha) f(y)^2 = \frac{f(x)^2}{f(x)^2+1} + \frac{f(y)^2f(x)^2}{f(x)^2+1} = 1,
	\]
	which concludes the proof.

\end{proof}

\begin{rmk}
	Further analysis of the above proof shows that Lemma~\ref{lem:KeyLemma} generalizes from metric graphs to connected compact Hausdorff spaces.
\end{rmk}

\begin{lem} \label{lem:AuxiliaryFormOfMinimizingMeasure}
	Consider a measure $\mu \in \cN$ with $\lambda_1(H_\mu) = \lambda_1^{\min}(G)$. Let $f$ be an eigenfunction for the first eigenvalue $\lambda_1(H_\mu)$ and let $\nu$ be the measure from Lemma~\ref{lem:KeyLemma}. Then $\supp\mu = \supp\nu =\{x,y\}$.

\end{lem}

\begin{proof}
	Since $\lambda_1(H_\mu) = \lambda_1^{\min}(G) < \infty$, the support of $\mu$ must contain at least two points. Thus, it suffices to prove that $\supp\mu \subseteq \supp\nu$.

	Assuming the opposite, let $x \in G$ be a point which lies in $\supp\mu$, but not in $\supp \nu$. Then there exists a small open neighborhood $U$ of $x$ such that $\nu(U) =  0$ and $0  < \mu(U) < 1$ (here we use that $\supp\mu  \neq \{x \}$).

	We claim that there exists a point $s \in U  \cap \supp \mu$ such that $f(s)\neq 0$. Indeed, if $f \equiv 0$ on $ U  \cap \supp \mu$, then the measure $\hat \mu \in \cN$ defined as
	\[
		\hat\mu(A) = \frac{1}{\mu(G \setminus U )} \mu (A \setminus U  ), \qquad \text{$A \subset  G$ measurable},
	\]
	would satisfy $\mu(G \setminus U ) \|f\|^2_{\hat \mu}  =  \|f\|^2_{\mu}$ and
	\begin{align*}
		\lambda_1(H_{\hat \mu}) = \min_{\substack{g \in H^1(G) \\g \perp_{\hat \mu} \one}} \frac{q(g)}{\|g\|_{\hat\mu}^2 }
		\le  \frac{q(f)}{\|f\|_{\hat\mu}^2 } < \frac{q(f)}{\|f\|_{\mu}^2 }  = \lambda_1(H_\mu)  = \lambda_1^{\min} (G),
	\end{align*}
	a contradiction.

	Fix a test function $\varphi \ge 0 $ in $H^1(G)$ with  $\varphi(s) > 0$ and which is supported on a small neighborhood $W$ of $s $, such that $f(w) \neq 0$ for all $w \in W$. Taking into account~\eqref{eq:RepresentationTheorem} and the fact that $s \in \supp\mu$, we obtain that
	\[
		q(f, \varphi) = \int_G  \varphi H_\mu f \, d\mu = \lambda_1(H_\mu) \int_W \varphi f \, d\mu \neq 0.
	\]
	We finish the proof by showing that,  simultaneously,
	\[
		q(f, \varphi) =0,
	\]
	which provides the desired contradiction. Since $\varphi$ is supported on $G \setminus \supp \nu$ by construction, it is enough to prove that $f$ is harmonic on $G  \setminus \supp \nu$. Consider the projection map $\cP_\nu \colon H^1(G) \to  H^1(G, \nu)$ for the orthogonal decomposition $H^1(G) = H^1(G, \nu) \oplus H^1_0(G \setminus \supp \nu)$. Taking into account the properties of $\nu$ in Lemma~\ref{lem:KeyLemma} and Proposition~\ref{Prop: Harmonic extension}, we obtain that
	\begin{align*}
		\lambda_1^{\min}(G) =
		\lambda_1 (H_\mu) & = \frac{q(f)}{\|f\|_\mu^2} = \frac{q(f)}{\|f\|_\nu^2} =  \frac{q(f)}{\| \cP_\nu(f)\|_\nu^2} \\
		                  & \ge  \frac{q(\cP_\nu f)}{\| \cP_\nu(f)\|_\nu^2} \ge \min_{\substack{g \in H^1(G, \nu)       \\ g\perp_\nu \one}} \frac{q(g)}{\| g \|_\nu^2} = \lambda_1(H_\nu)  \ge \lambda_1^{\min}(G).
	\end{align*}
	These inequalities show that  $q(f)  = q(\cP_\nu(f))$. Applying Proposition~\ref{Prop: Harmonic extension} once again, we conclude that $f$ is harmonic on $G \setminus \supp \nu$, finishing the proof.

\end{proof}

After these preparations, we are in position to prove Theorem~\ref{thm: characerisation of lambda1}.

\begin{proof}[Proof of Theorem~\ref{thm: characerisation of lambda1}]
	Let $\mu \in \cN$ and fix an eigenfunction $f$ for the first eigenvalue $\lambda_1(H_\mu)$. Consider the discrete measure $\nu= \alpha \delta_x +(1-\alpha) \delta_y$ constructed in Lemma~\ref{lem:KeyLemma}. We will prove that
	\begin{equation} \label{eq:InequalityChainExistenceMinimizingMeasure}
		\lambda_1(H_\mu) \ge \lambda_1(H_\nu) \ge \lambda_1(H_{\mu_{x,y}}).
	\end{equation}
	The first inequality in~\eqref{eq:InequalityChainExistenceMinimizingMeasure} is a consequence of the properties of $\nu$, that is,
	\[
		\lambda_1(H_{\mu})= \frac{q(f)}{\|f\|^2_{L^2(G, \mu)}} = \frac{q(f)}{\|f\|^2_{L^2(G, \nu)}}  \ge \inf_{\substack{g \in H^1(G)\setminus \{0\} \\ g \perp_{\nu} \one}} \frac{q(g)}{\|g\|^2_{L^2(G, \nu)}} = \lambda_{1}(H_\nu),
	\]
	where we used~\eqref{eq:MinLambda1}. Consider the family of measures $\nu_\beta := \beta \delta_x + (1-\beta) \delta_y$ for $\beta \in (0,1)$. We prove that
	\begin{equation} \label{eq:CoefficientsEstimate}
		\lambda_1(H_{\nu_\beta}) \ge \lambda_1(H_{\mu_{x,y}})
	\end{equation}
	for all $\beta$, which in particular implies the second inequality in~\eqref{eq:InequalityChainExistenceMinimizingMeasure}. Fix a $\nu_\beta$-normalized eigenfunction $f_\beta \in H^1(G, \nu_\beta)$ for $\lambda_1(H_{\nu_\beta})$. Since $f_\beta \perp_{\nu_\beta} \one_G$, we infer
	\[
		f_\beta (y) = \frac{\beta}{\beta -1} f_\beta(x).
	\]
	Using this equality in the equation $\|f_\beta\|_{L^2(G,\mu_\alpha)} =1$ leads to
	\[
		f_\beta(x) = \pm \sqrt{\frac{1-\beta}{\beta}}.
	\]
	By Proposition~\ref{Prop: Harmonic extension}, $f_\beta$ is now uniquely determined by $\beta$ (up to the sign). Taking $g \in H^1(G)$ as the unique function which is harmonic on $G\setminus \{x,y\}$ and satisfies $g(x) = 1$ and $g(y) = 0$, we obtain the equality $f_\beta - f_\beta(y)\one= (f_\beta(x) - f_\beta(y)) g$. This means that the minimum
	\[
		\min_\beta  \lambda_1(H_{\nu_\beta})  = \min_\beta q(f_\beta) =  \min_\beta (f_\beta(x) - f_\beta(y))^2 q(g) = q(g) \min_\beta \left|\sqrt{\frac{1-\beta}{\beta}} + \sqrt{\frac{\beta}{1-\beta}} \right|^2
	\]
	is uniquely attained for $\beta = 1/2$, so that~\eqref{eq:CoefficientsEstimate} follows.

	Having established~\eqref{eq:InequalityChainExistenceMinimizingMeasure}, we prove the above claims as follows. Let $\mu \in \cN$ be a measure with $\lambda_1(H_\mu)  = \lambda_1^{\min}$. Then~\eqref{eq:InequalityChainExistenceMinimizingMeasure} implies that $\lambda_1(H_{\mu_{x,y}}) = \lambda_1(H_\mu) = \lambda_1^{\min}$. On the other hand, by Lemma~\ref{lem:AuxiliaryFormOfMinimizingMeasure}, the support of $\mu$ equals $\supp \mu = \supp \nu=\{x,y\}$, so that $\mu = \nu_\beta$ for some $\beta \in (0,1)$. Since the inequality in~\eqref{eq:CoefficientsEstimate} is strict for $\beta \neq1/2$, it follows that $\mu  = \mu_{x,y}$, completing the proof.
\end{proof}

Theorem~\ref{thm: characerisation of lambda1} implies another characterization of $\lambda_{1}^{\min}(G)$ using the space $H^1(G)$.
\begin{cor}\label{Cor: Alternative charakterization of lambda 1}
	Let $G$ be a metric graph. Then
	\[
		\lambda_{1}^{\min}(G) = \min_{\substack{f \in H^1(G) \\ \max f =1, \, \min f= -1}}q(f).
	\]
\end{cor}
\begin{proof}
	Consider first a function $f\in H^1(G)$ such that $f(a)=1 = \max f$ and $f(b)=-1=\min f$ for some $a\neq b \in G$. Then Proposition~\ref{Prop: Harmonic extension} implies that
	\[
		q(f) \ge q(f_{a,b}) = \lambda_1 (H_{\mu_{a,b}}) \ge \lambda_{1}^{\min}(G).
	\]
	On the other hand, let $\mu  = \mu_{x,y}$ be a minimizing measure from Theorem~\ref{thm: characerisation of lambda1}. Then the corresponding eigenfunction $f_{x,y} \in H^1(G, \mu)$ satisfies $\max f = f(x) = 1$ and $\min f = f(y) = -1$ by Proposition~\ref{Prop: Max and Min in harmonic functions}. Moreover, $\lambda_{1}^{\min}(G) = \lambda_1(H_\mu) = q(f_{x,y})$. Combining these observations, the claim follows.
\end{proof}

For later use, we also collect the following auxiliary fact.

\begin{lem}\label{lem: energy between two points} Let $\mu = \mu_{x,y}$ be a minimizing measure for $\lambda_{1}^{\min}(G)$. Then
	\[
		\lambda_{1}^{\min}(G)  = q(f_{x,y}) =  2 \sum_{e=xz} \frac{1-f_{x,y}(z)}{\ell(e)}  =  2 \sum_{e=yz} \frac{1+ f_{x,y}(z) }{\ell(e)},
	\]
	where in this representation we consider a model $(V,E, \ell)$ of $G$ with $x,y \in V$.
\end{lem}

\begin{proof} Since $f_{x,y}$ is harmonic on $G\setminus\{x,y\}$, Proposition~\ref{Prop: Harmonic extension} leads to
	\begin{align*}
		0=q(f_{x,y}, \one_G) =  \sum_{e =xz } \frac{1-f_{x,y}(z)}{\ell(e)} + \sum_{e=yz} \frac{-1-f_{x,y}(z)}{\ell(e)} ,
	\end{align*}
	which proves the last of the above equalities. Applying Proposition~\ref{Prop: Harmonic extension} once again,
	\begin{align*}
		q(f_{x,y}) = \sum_{e = xz} \frac{1-f_{x,y}(z)}{\ell(e)} + \sum_{e=yz} \frac{1+f_{x,y}(z)}{\ell(e)},
	\end{align*}
	yielding the remaining equalities.
\end{proof}

\subsection{Characterization via resistance} \label{ss:CharacterizationResistance}

Let $G$ be a metric graph. For points $x,y \in G$, there exists a unique solution $f\in H^1 (G)$ to the equations
\begin{equation}\label{eq:SystemEquationsResistance}
	\begin{cases}
		\Delta f = \delta_x - \delta_y \\
		f(y) =0
	\end{cases}
\end{equation}
where $\Delta$ is the distributional Laplacian on $H^1(G)$, i.e., the first equation in~\eqref{eq:SystemEquationsResistance} is understood as
\[
	q(f,g) = g(x) - g(y) \qquad \text{for all $g \in H^1(G)$}.
\]

The \emph{(effective) resistance} $r(x,y)$ between $x$ and $y$ is defined as
\[
	r(x,y) =f(x).
\]
The reciprocal $1/r(x,y)$ is sometimes called the \emph{(effective) conductance} between $x$ and $y$.

In the literature, one can find many equivalent definitions of the resistance and in this regard we refer to~\cite[Theorem~3.2]{jp23}. For instance,
\begin{equation}\label{eq: resistance characerisation}
	r(x,y) = 1/ \min\{ q(f) | \, f \in H^1(G), f(x) = 1, f(y) = 0 \}.
\end{equation}
We need also the following representation in terms of the function $f_{x,y}$ introduced in~\eqref{eq:Definitionfxy}. Taking into account Proposition \ref{Prop: Harmonic extension} and Lemma \ref{lem: energy between two points}, it follows that the unique solution to \eqref{eq:SystemEquationsResistance} is given by $f = \frac{2}{q(f_{x,y})} (f_{x,y}+\one_{G})$. In particular, the resistance $r(x,y)$ can be written as
\begin{equation} \label{eq:RepresentationResistance}
	r(x,y) = \frac{2}{q(f_{x,y})} (f_{x,y}(x) +1)  = \frac{4}{q(f_{x,y})}.
\end{equation}

It is well-known that the resistance $(x,y)\in G\times G \mapsto r(x,y)$ is a metric on $G$ (e.g.,~\cite[Lemma~3.5]{jp23}). It is comparable to the path metric $\varrho$ in the sense that (see, e.g.,~\cite[Lemma~3.52]{jp23})
\begin{equation} \label{eq:InequalityResistancePathMetric}
	r(x,y) \le \varrho(x,y) \qquad \text{for all $x,y \in G$}.
\end{equation}
Moreover, the metrics $r$ and $\varrho$ coincide if and only if $G$ is a tree. For general graphs, the next lemma controls the resistance $r$ using the square of the path metric $\varrho$.
\begin{lem} \label{lem:EasyUpperEstimate}
	Let $G$ be a metric graph with total length $L(G)$. Then
	\[
		r(x,y)\ge \frac{\varrho(x,y)^2}{L(B_\varrho(x, \varrho(x,y))} \ge \frac{\varrho(x,y)^2}{L(G)} \qquad \text{for all $x,y\in G$}.
	\]
	Here, $B_\varrho(x, t)$ denotes the ball in the path metric $\varrho$ of radius $t$ around $x \in G$, and $L(B_\varrho(x,t))$ its Lebesgue volume.
\end{lem}
\begin{proof}
	By~\eqref{eq: resistance characerisation}, the estimate is equivalent to the inequality
	\[
		\min\{ q(f) | f \in H^1(G) \text{ with $f(x) = 1$ and $f(y) = 0$} \} \le \frac{L(B_\varrho(x, \varrho(x,y))}{\varrho(x,y)^2}.
	\]
	The latter follows immediately by considering the function $f\colon G \to \R$ defined by
	\[
		f(s) =\max\{1 - \frac{\varrho(x,s)}{\varrho(x,y)}, 0\},  \qquad s \in G.
	\]
	The derivative $f'$ is supported on $B_\varrho(x, \varrho(x,y))$, where it satisfies $|f'| = 1/\varrho(x,y)$ almost everywhere. Estimating $q(f)$ accordingly, we complete the proof.
\end{proof}
\begin{rmk}
	The metrics $\varrho$ and $r$ are actually bilipschitz equivalent with
	\[
		r(x,y) \le \varrho(x,y) \le 2 \#(E) \, r(x,y) \qquad \text{for all $x,y \in G$}.
	\]
	This follows from~\eqref{eq:InequalityResistancePathMetric} combined with Lemma~\ref{lem:EasyUpperEstimate} and the elementary estimate $L(B_\varrho(x, t)) \le 2 \#(E) t$, which is valid for $x \in G$ and $t > 0$.
\end{rmk}

The next theorem characterizes $\lambda_1^{\min}(G)$ in terms of the resistance metric $r$.
\begin{thm} \label{Lem: Resitance metric and optimal eigenvalue}
	Let $G$ be a metric graph and $\diam_r(G)$ the diameter of $G$ with respect to the resistance metric $r$, i.e., the maximum of $r\colon G \times G \to \R_{\geq 0}$. Then
	\[
		\lambda_1^{\min}(G) = \frac{4}{\diam_r(G)}.
	\]
	Moreover, the measures $\mu \in \cN$ with $\lambda_1(H_\mu) = \lambda_1^{\min}(G)$ are precisely the measures $\mu_{x,y}$ for points $x, y\in G$ with $r(x,y) = \diam_r(G)$.

	If $G$ is a tree, then
	\[
		\lambda_1^{\min}(G) = \frac{4}{\diam_r(G)} = \frac{4}{\diam_\varrho(G)}.
	\]
\end{thm}
\begin{proof}
	Combining~\eqref{eq:ExpressLamda1Fxy} with the representation~\eqref{eq:RepresentationResistance} for $r$, we infer that
	\[
		\lambda_1^{\min}(G)  = \min_{x \neq y} \lambda_1 (H_{\mu_{x,y}}) =\min_{x \neq y} q(f_{x,y}) =  \min_{x \neq y} \frac{4}{r(x,y)} = \frac{4}{\diam_r(G)}.
	\]
	The claim concerning trees follows from the fact that $r  = \varrho$ for trees.
\end{proof}

\subsection{Characterization via spectral partitioning}
As it turns out, there is an interesting connection between the first optimal eigenvalue $\lambda_1^{\min}$ and spectral partitioning on graphs.
We begin with a variant of optimal eigenvalues which takes into account a Dirichlet condition.

Throughout this section, let $G=(V,E, \ell)$ be a metric graph. Given a closed, non-empty subset $A \subset G$, we
denote by
\[
	H^1_0(G,A) := \{f \in H^1(G) \; | \; f|_{A} \equiv 0\}
\]
the space of $H^1$-functions vanishing in $A$. For a measure $\mu\in \cM(G \setminus A)$, that is, a measure $\mu \in \cM$ with $\mu(A) = 0$, we introduce the Sobolev space
\begin{align*}
	H^1_0(G,A,\mu) := \{f \in H^1_0(G,A) \; | \; f \text{ harmonic on } G\setminus (\supp \mu \cup A)\},
\end{align*}
which is a dense subspace of the $L^2$-space $L^2(G, \mu)$. In the above context, we also refer to $A$ as the \emph{Dirichlet set}.

It turns out that the restriction of the energy form $q$ to $H^1_0(G,A,\mu)$ is closed. Replacing $H^1(G,\mu)$ by $H^1_0(G,A, \mu)$ in the construction of the Laplacian $H_\mu$ leads to a self-adjoint operator
\[
	H_\mu^D \colon \dom(H_\mu^D) \subset L^2(G, \mu) \to L^2(G, \mu)
\]
in the Hilbert space $L^2(G, \mu)$. We call $H_\mu^D$ the \emph{Laplacian for $\mu$ with Dirichlet condition at $A$}, or simply the \emph{Dirichlet Laplacian for $\mu$}.

The operator $H_\mu^D$ is non-negative and has purely discrete spectrum. Taking into account multiplicities, we enumerate its eigenvalues as
\[
	\lambda_0(H_\mu^D) \le \lambda_1(H_\mu^D) \le \lambda_2(H_\mu^D) \le \dots.
\]
Note that the smallest eigenvalue $\lambda_0(H_\mu^D)$ is strictly positive and admits the variational characterization
\begin{equation} \label{eq:VariationalCharacterizationDirichlet}
	\lambda_0(H_\mu) = \inf_{\substack{f \in H^1_0(G, A, \mu) \\ f \neq 0}} \frac{q(f)}{\|f\|^2_{L^2(G, \mu)}} = \inf_{\substack{f \in H^1_0(G, A) \\ f \neq 0}} \frac{q(f)}{\|f\|^2_{L^2(G, \mu)}}.
\end{equation}
We are interested in the following Dirichlet variant of optimal eigenvalues.
\begin{defn}
	Let $A$ be a closed, non-empty subset in a metric graph $G$. For $k \in \N_0$, the \emph{$k$-th optimal Dirichlet eigenvalue with respect to $A$} is defined by
	\[
		\lambda_k^{\min,D} :=  \lambda_k^{\min,D}(G, A) := \inf_{\mu \in \cN(G \setminus A)} \lambda_{k}(H_\mu^D),
	\]
	where $\cN(G \setminus A) := \{\mu \in \cN| \, \mu(A) = 0\}$.
\end{defn}
We focus on the smallest optimal Dirichlet eigenvalue $\lambda_0^{\min,D}$. A first estimate is provided by the following variant of Lemma~\ref{Prop: Nicaise}.
\begin{lem}\label{Prop: Nicaise2}
	Let $G$ be a metric graph and $\emptyset \neq A\subset G$ a closed subset. Then
	\begin{equation}\label{eq: Dirichlet Nicaise2}
		\lambda_{0}^{\min, D}(G, A) \geq \frac{1}{\diam(G)},
	\end{equation}
	where $\diam(G)= \diam_\varrho(G)$ is the diameter of $G$ with respect to the path metric $\varrho$.
\end{lem}
\begin{proof}
	This follows from Lemma~\ref{Prop: Nicaise} and the characterization of $\lambda_0(H_\mu)$ in~\eqref{eq:VariationalCharacterizationDirichlet}.
\end{proof}

The next lemma reduces the study of $\lambda_{0}^{\min, D}(G,A)$ to the case where $G \setminus A$ is connected. For a subset $A \subset G$ as above, the closures of all connected components $U_i$, $i \in I$, of $G \setminus A$ form a countable family $(\overline U_i)_{i \in I}$ of metric subgraphs of $G$.

\begin{lem} \label{lem:MinimizeDirichletOverConnectedComponents}
	Notations as above,
	\[
		\lambda_{0}^{\min, D}(G, A) = \min_{i \in I} \lambda_{0}^{\min, D}(\overline{U_i} , \partial U_i),
	\]
	where $\partial U_i =  \overline{U_i} \cap A$ denotes the boundary of $U_i$ in $A$.
\end{lem}
\begin{proof} First we prove that
	\[
		\lambda_{0}^{\min, D}(G, A) = \inf_{i \in I} \lambda_{0}^{\min, D}(\overline{U_i}, \partial U_i).
	\]
	The inequality \glqq$\leq$\grqq{} is trivial, since every measure on a connected component $\overline{U_i}$ yields a measure on $G$. To prove the reverse inequality, take a measure $\mu \in \cN(G \setminus A)$ and consider the restrictions $\mu_i = \mu|_{\overline{U_i}}$ for $i \in I$. Then $\lambda_0 (H_\mu^D)$ is an eigenvalue for the Laplacian $H_{\mu_i}^D$ for every component $\overline{U_i}$ on which the corresponding eigenfunction $f \in L^2(G, \mu)$ is non-vanishing. Fixing such a component $\overline{U_i}$ yields
	\[
		\lambda_{0}(H_{\mu}^D) \geq \lambda_{0}(H_{\mu_i}^D)  \geq \frac{1}{\mu(\overline{U_i})}\lambda_{0}^{\min, D}(\overline{U_i}, \partial U_i).
	\]
	Since $\mu(\overline{U_i})\leq 1$, the desired estimate follows.

	Let $C> 0$ be arbitrary. Assume there exists a countable infinite set $J\subset I$ with $\lambda_{0}^{\min, D}(\overline{U_i}, \partial U_i) < C$ for every $i \in J$. Applying Lemma~\ref{Prop: Nicaise2}, we obtain
	\[
		L(G) \geq \sum_{i \in J} L(\overline{U_i}) \geq \sum_{i \in J} \diam(\overline{U_i}) \geq \sum_{i \in J} \frac{1}{\lambda_{0}^{\min D}(\overline{U_i}, \partial U_i)} \geq  \sum_{i \in J} \frac{1}{C} = \infty.
	\]
	This contradiction shows that the infimum is in fact a minimum.
\end{proof}

The following result provides an analog of Theorem~\ref{thm: characerisation of lambda1} for optimal Dirichlet eigenvalues.

\begin{lem} \label{lem:MinimizingMeasuresDirichletEigenvalue} Let $G$ be a metric graph and $\emptyset \neq A\subset G$ a closed subset. Then there exists a point $x \in G\setminus A$ such that
	\[
		\lambda_{0}^{\min, D}(G, A) = \lambda_0 (H_{\delta_x}^D).
	\]
	Moreover, every measure $\mu \in \cN(G\setminus A)$ with $\lambda_0(H_\mu^D) = \lambda_{0}^{\min, D}(G, A)$ is a Dirac measure $\mu = \delta_x$ for some $x \in G \setminus A$.
\end{lem}
\begin{proof} By Lemma~\ref{lem:MinimizeDirichletOverConnectedComponents}, it suffices to treat the case where $A$ is finite and $G \setminus A$ is connected. Fixing a suitable model of $G$, we may also assume that all points of $A$ belong to the vertex set $V$. Using the min-max principle, we obtain that
	\begin{align*}
		\lambda^{\min, D}_0 (G,A) & = \inf_{\mu \in \cN(G \setminus A)} \inf_{\substack{f \in H^1_0(G,A) \\ \|f\|_\infty =1}} \frac{q(f)}{\int_{G} |f|^2 d\mu} \geq \inf_{\mu \in \cN(G \setminus A)} \inf_{\substack{f \in H^1_0(G,A) \\ \|f\|_\infty =1}} q(f) \\
		                          & =  \inf_{\substack{f \in H^1_0(G,A)                                  \\ \|f\|_\infty =1}} q(f) = \inf_{\substack{f \in H^1_0(G,A) \\ f \ge0, \, \max f = 1}} q(f).
	\end{align*}
	For every $f \ge 0$ in $H^1_0(G,A)$ with $\max f =1$, fix a point $x_f \in G$ with $f(x_f)=1$. Consider the unique function $\widetilde{f} \in H^1(G)$ such that
	\[
		\begin{cases}
			\widetilde{f} (x) = f(x), & \text{ if } x \in V \cup \{x_f\},         \\
			f \text{ harmonic },      & \text{ on } G\setminus (V \cup \{x_f\} ).
		\end{cases}
	\]
	We have $q(\widetilde f) \le q(f)$ by Proposition~\ref{Prop: Harmonic extension} and $\widetilde{f} \in H_0^1(G,A)$, since $A \subset V$. Setting $F := \{\widetilde{ f} \; | \; f \in H^1_0 (G,A), \, f \ge 0, \, \max f = 1 \}$, the above leads to
	\[
		\lambda^{\min, D}_0(G,A)  \geq \inf_{h \in F} q(h).
	\]

	Consider a sequence $(\widetilde f_n )_n \subset F$ such that $q( \widetilde f_n ) \to \inf_{h \in F} q(h)$ for $n \to \infty$. Since $0 \le \widetilde f_n \le 1$ and $G$ is compact, we may suppose that the values $\widetilde f_n (v)$ for $v \in V$ have limits $\alpha_v \in \R$, and moreover the corresponding points $x_{f_n} \in G$ converge to a point $x_0 \in G$. We would like to define a function $g \in H^1(G)$ by requiring that
	\[
		\begin{cases}
			g(v) = \alpha_v,     & \text{for all $v \in V$},                 \\
			g(x_0) = 1,                                                      \\
			g \text{ harmonic }, & \text{ on } G\setminus (V \cup \{x_0\} ).
		\end{cases}
	\]
	In order to prove that $g$ is well-defined, we have to verify that, if $x_0 = u$ is a vertex, then we also have $\alpha_u = 1$. However, this follows from the estimate
	\begin{align*}
		|\alpha_u - 1|^2 = \lim_{n \to\infty} | \widetilde f_n (u) - \widetilde f_n (x_{f_n}) |^2 \le  \lim_{n \to\infty} \varrho(u, x_{f_n}) q(  \widetilde f_n ) = 0.
	\end{align*}
	By construction, the function $g$ belongs to $H^1_0(G,A)$. Using that the energy of an edgewise linear function $h \in H^1(G)$ equals $q(h) = \sum_{e = vw} (h(v)-h(w))^2/\ell(e)$, one can further verify that $\lim_{n \to \infty} q( \widetilde f_n ) \ge q(g)$.
	Setting $\mu= \delta_{x_0}$, we then obtain that
	\[
		\lambda^{\min, D}_0(G,A) \ge  \lim_{n \to \infty} q( \widetilde f_n ) \ge q(g) = \frac{q(g)}{\|g\|^2_{L^2(G, \mu)}} \ge \lambda_0 (H_{\mu}^D) \ge \lambda^{\min, D}_0(G,A)
	\]
	and the first statement is proved.

	Assume that $\mu \in \cN(G\setminus A)$ satisfies $\lambda_0(H_\mu^D) = \lambda_{0}^{\min, D}(G,A)$. Let $f$ be a non-negative eigenfunction for $\lambda_0(H_\mu^D)$ with $\max_{x\in G} f(x) = 1$. With Proposition \ref{Prop: Max and Min in harmonic functions}, we can choose $x \in \supp \mu\setminus A$ with $f(x)=1$. Further let $g \in H^1_0(G, A)$ be the unique solution to the equations
	\[
		\begin{cases}
			g(x) = 1,                 \\
			g \equiv 0\text{ on $A$}, \\
			g \text{ harmonic on } G\setminus (A\cup \{x\}).
		\end{cases}
	\]
	Due to Proposition \ref{Prop: Harmonic extension}, we have
	\[
		\lambda_{0}^{\min, D}(G,A) = \lambda_0(H_\mu^D) = \frac{q(f)}{\int_G |f|^2 d\mu} \geq q(g) = \lambda_0(H_{\delta_x}^D) \geq \lambda_0^{\min, D}(G,A). \]
	This equality means that $g = f$ and $\int_G |f|^2 d\mu = 1$. In particular, $f= 1$ $\mu$-almost everywhere. But $f=g$ takes it maximum only at $x$ due to Proposition~\ref{Prop: Max and Min in harmonic functions}. It follows that $\mu = \delta_x$, and the proof is complete.
\end{proof}

\begin{rmk}
	The infimum in the definition of $\lambda_k^{\min,D}$ is in fact attained for all $k \in \N_0$. This can be shown by arguments similar to the proof of Theorem~\ref{thm: spec conv }.
\end{rmk}

Our next objective is to introduce a notion of graph partitions adapted to our context. A (not necessarily exhaustive) \emph{weighted $n$-partition} of a metric graph $G$ is a collection
\[
	\mathcal{P}= \big \{ (G_1,\omega_1), \dots, (G_n, \omega_n) \big \}
\]
where $G_1, \dots, G_n$ are closed, non-empty and connected subsets of $G$ with mutually disjoint interiors, and $\omega_1, \dots \omega_n$ are strictly positive reals with $\sum_{i=1}^n \omega_i=1$.
Note that we only use the terminology partition informally, since the sets $G_i$ may intersect on their boundaries, and their union might be a proper subset of $G$.
The \emph{set of weighted $n$-partitions of a metric graph $G$} is denoted by $\mathfrak{P}_n(G)$.

Given a weighted partition $\cP \in \mathfrak{P}_n(G)$, every closed set $G_i$ in $\cP$ is itself naturally a metric graph. We endow $G_i$ with the Dirichlet set $\partial G_i$, that is, the topological boundary of $G_i$ in $G$. Then $\cP$ has an inherent \emph{Dirichlet energy} given by
\[
	\Lambda (\cP) := \max_{i =1, \dots,n} \frac{1}{\omega_i}\lambda_{0}^{\min, D} (G_i, \partial G_i).
\]
Minimizing the energy with respect to all partitions leads to
\[
	\mathcal{L}_{n} (G) := \inf_{\cP \in \mathfrak{P}_n} \Lambda (\cP).
\]
Finding a minimizing partition, i.e.,~a partition $\cP \in \mathfrak{P}_n(G)$ with $\Lambda (\cP) = \mathcal{L}_{n} (G)$, is a hard problem in general. However, for partitions consisting only of $n=2$ pieces, we have the following connection to the first optimal eigenvalue.
\begin{thm} \label{thm:SpectralPartitioning}
	Let $G$ be a metric graph. Then
	\[
		\mathcal{L}_2(G) = \lambda_{1}^{\min}(G).
	\]
	Moreover, if $\mu \in \cN$ is a minimizing measure for $\lambda_1^{\min}(G)$ and $f \in L^2 (G, \mu)$ is an eigenfunction for $\lambda_1(H_\mu)$, then the nodal domains
	\begin{align*}
		N^+= \overline{\{x \in G\;|\; f(x)> 0\}} \qquad \text{and} \qquad N^-= \overline{\{x \in G\;|\; f(x) < 0\}}
	\end{align*}
	together with the weights $\omega_+ = \omega_- = \frac{1}{2}$ form a minimizing $2$-partition for $\mathcal{L}_2(G)$.
\end{thm}
\begin{proof}
	Let $\mu = \mu_{x,y}$ be a minimizing measure from Theorem~\ref{thm: characerisation of lambda1}, and $f=f_{x,y}$ the corresponding eigenfunction. Then the nodal domains $N^+$ and $N^-$ are closed, non-empty and have disjoint interiors. Their connectivity can be deduced from the min-max principle. That is, enumerate the connected components of $N^+$ and $N^-$ as $N_1^+, \dots, N_p^+$ and $N_1^-, \dots, N_q^-$, respectively, with $x \in N^+_1$ and $y \in N^-_1$. Then by Corollary \ref{Cor: Alternative charakterization of lambda 1},
	\[
		\lambda_{1}(H_\mu) = q(f) \geq q(f|_{N^+_1 \cup N^-_1}) \geq \lambda_{1}(H_\mu),
	\]
	so $f$ is constant on each connected component of $G\setminus (N^+_1 \cup N^-_1)$. Due to the continuity of $f$, this means that $f \equiv 0$ on $G\setminus (N^+_1 \cup N^-_1)$, and thus $p=q =1$.

	It follows that $\cP=\{(N^+, 1/2),(N^-, 1/2)\}$ is a weighted $2$-partition. Note that the restriction $f|_{N^+}$ is an eigenfunction for the operator $H_{\delta_x}^D$ on $N^+$ with eigenvalue $\lambda_{1}(H_\mu)/2$. The same holds for $\delta_{y}$ on $N^-$. Thus,
	\[
		\lambda_{1}^{\min}(G) = \lambda_{1}(H_\mu) \ge \Lambda (\cP) \geq \mathcal{L}_2 (G).
	\]

	Conversely, let $\cP=\{(G_1,\omega_1), (G_2, \omega_2)\}$ be an arbitrary weighted $2$-partition. By Lemma~\ref{lem:MinimizingMeasuresDirichletEigenvalue}, we can find points $x_i \in G_i $, $i=1,2$, such that the measure $\mu_i := \delta_{x_i}$ minimizes the first  Dirichlet eigenvalue on $G_i$. Set $\mu := \omega_1 \mu_1 + \omega_2 \mu_2 \in \cN$. Fix a $\mu_i$-normalized, non-negative eigenfunction $f_i \in H_0^1(G_i, \partial G_i)$ for $\lambda_0(H_{\mu_i}^D)$ on $G_i$. We extend $f_i$, $i=1,2$, by zero to all of $G$ and consider the function
	\[
		f:=\sqrt{\frac{\omega_2}{\omega_1}} f_1-\sqrt{\frac{\omega_1}{\omega_2}}f_2 \in H^1(G).
	\]
	By construction, $f \perp_{\mu} \one_{G}$ and $\|f\|_\mu =1$, so that
	\begin{align*}
		\lambda_{1}^{\min}(G) \leq \lambda_{1}(H_\mu) \leq q(f) & = \frac{\omega_2}{\omega_1}   \int_{G_1} |f_1'|^2dx + \frac{\omega_1}{\omega_2}\int_{G_2} |f_2'|^2dx                                 \\
		                                                        & = \omega_2 (\frac{1}{\omega_1}\lambda_{0}(H_{\mu_1}^D)) + \omega_1(\frac{1}{\omega_2}\lambda_{0}(H_{ \mu_2 }^D)) \leq  \Lambda(\cP).
	\end{align*}
	Since $\cP\in \mathfrak{P}_2(G)$ was arbitrary, we have proved the statement.
\end{proof}

The analog of Theorem~\ref{thm:SpectralPartitioning} for higher optimal eigenvalues fails. For example, $\lambda_{2}^{\min}(C) < \mathcal{L}_{3}(C)$ for the circle $C$. Nevertheless, we pose the following conjecture, which constitutes a weaker version of Theorem~\ref{thm:SpectralPartitioning} for $k \ge 1$.

\begin{conjecture}
	Let $\cP = (G_i, \omega_i)_{i=0, \dots, k}$ be a minimizing partition for $\mathcal{L}_{k+1}(G)$ and $x_i$, $i=0, \dots, k$, be points on $G$ such that $\lambda_0^{\min, D}(G_i, \partial G_i)$ is realized by the Dirac measure $\delta_{x_i}$. Then the measure
	\[
		\mu = \sum_{i=0}^{k} \omega_i \delta_{x_i}
	\]
	is minimizing, i.e.,
	\[
		\lambda_{k}^{\min}(G) = \lambda_{k}(H_\mu).
	\]
	In particular, there is a minimizing measure for $\lambda_{k}^{\min}(G)$ supported on $k+1$ points.
\end{conjecture}

We close the section with an overview of all equivalent characterizations.
\begin{thm}
	Let $G$ be a metric graph. Then:
	\[
		\lambda_{1}^{\min}(G) = \min\limits_{\mu \in \cN} \lambda_{1}(H_\mu) = \min\limits_{\substack{x,y \in G \\ x\neq y}} q(f_{x,y})  = \min\limits_{\substack{f \in H^1(G) \\ \max f =1\\ \min f= -1}} q(f) = \dfrac{4}{\diam_r(G)}  = \mathcal{L}_{2}(G).
	\]
\end{thm}


\section{Further properties of the first optimal eigenvalue} \label{sec:PropertiesFirstEigenvalue}
In what follows, we discuss further properties of the first optimal eigenvalue $\lambda_1^{\min}$. Throughout this section, let $G = (V,E, \ell)$ be a metric graph.

\subsection{Basic properties}
The next result further specifies the minimizing measures for $\lambda_1^{\min}(G)$. Recall that $\deg(s)$ denotes the degree of a point $s \in G$, with $\deg(s) := 2$ for all $s \in G \setminus V$.

\begin{thm}\label{thm: points avoid vertices}
	Let $\mu \in \cN$ be a measure with $\lambda_{1}^{\min}(G)=\lambda_1(H_{\mu})$ and write $\mu$ as $\mu = \mu_{x,y} = (\delta_x + \delta_y)/2$ for two points $x, y \in G$. Then $\deg(x), \deg(y) \in \{1,2\}$, i.e.~the two points lie at the end of a pending edge or in the interior of an edge.
\end{thm}

\begin{proof}
	By Theorem~\ref{thm: characerisation of lambda1}, the minimizing measure $\mu$ is of the form $\mu = \mu_{x,y}$ for some points $x, y\in G$. Upon refining the model, we may assume that $x$ and $y$ are vertices.

	Assume that $\deg(x) > 2$ and let $f=f_{x,y}$. First contract every edge, where $f$ is constant. The new graph $G'$ has a bigger first optimal eigenvalue by Lemma~\ref{lem:DecreasingEdgeLength}. But we have $q(f) = q(f|_{G'})$, so the first optimal eigenvalue in fact coincides by Corollary~\ref{Cor: Alternative charakterization of lambda 1}. This also shows that $\mu_{x,y}$ is a minimizing measure for $G'$. Thus, we may assume that $G=G'$, that is, $f$ is non-constant on every edge.

	Since the function $f$ is linear and non-constant on edges, all its level sets $\{z \in G\; | \; f(z)=c\}$, $ c \in \R$, are finite. Denote the values, which are taken by $f$ in vertices, by $a_1 = 1 = f(x) > a_2 > \dots > a_n = -1 = f(y)$. We assume that for each $i$, the level set $V_i := \{z \in G | f(z) = a_i\}$ consists only of vertices. This can always be achieved by refining the model. In particular, $V = \bigsqcup_i V_i$.

	We construct a new graph $\widetilde G$ by glueing together all vertices $v \in V_i$ for each set $V_i$. Then $\widetilde G$ is a pumpkin chain with vertices $\widetilde w_1, \dots \widetilde w_n$ corresponding to the values $a_1, \dots a_n$. The points $x$ and $y$ can be identified with the endpoints $\widetilde w_1$ and $\widetilde w_n$ of $\widetilde G$, so that the pumpkin chain $\widetilde G$ starts with a $\deg(x)$-pumpkin and ends with a $\deg(y)$-pumpkin.

	The glueing induces a natural surjective map $\Phi\colon G \to \widetilde{G}$. By construction, the function $f$ is the pullback $f = \widetilde{ f} \circ \Phi $ of a function $\widetilde f \in H^1(\widetilde G)$ with $q(f) = q(\widetilde{ f})$. On the other hand, for every function $g \in H^1(\widetilde{G})$, the pullback $g \circ \Phi \in H^1(G)$ has the same energy, maximum and minimum. By Corollary~\ref{Cor: Alternative charakterization of lambda 1}, we infer
	\[
		q(\widetilde{ f})	\geq \lambda_{1}^{\min}(\widetilde{G}) \geq \lambda_{1}^{\min}(G) = q(f) = q(\widetilde{ f}).
	\]
	Applying Proposition \ref{prop: surgery}, we can collapse every pumpkin in $\widetilde G$ except the first one into a path of suitable length, to obtain a graph $\widehat{G}$ with the following property: for the corresponding surjective mapping $\Psi\colon \widetilde{G} \to \widehat{G}$, there exists a unique function $\widehat{f}$ with $\widehat{f} \circ \Psi = \widetilde{ f}$ and $q(\widehat{f}) = q(\widetilde{ f})$. The same argument as above implies that
	\[
		\lambda_{1}^{\min}(\widehat{G}) = q(\widehat{f}).
	\]

	We proceed by calculating the energy $q(\widehat{f})$. The graph $\widehat{G}$ contains a $\deg(x)$-pumpkin, having $x$ as its left endpoint and a vertex $z$ as its right endpoint. Either $z=y$, and $\widehat G$ consist only of the pumpkin, or it contains an additional edge $e_y$ attached at $z$, having $y$ as its other endpoint. Let $\ell_1\geq \dots\geq \ell_{\deg (x)} >0$ be the lengths of the edges $e_1, \dots, e_{\deg(x)}$ in the pumpkin at $x$ and let $\ell_0$ be the length of the remaining graph. If the edge $e_y$ is absent, we have $\ell_0 = 0$. Further set
	\[
		S := \sum_{i=2}^{\deg(x)} \frac{1}{\ell_i}.
	\]
	In terms of these quantities, the energy of $\widehat f$ equals
	\[
		q(\widehat{f}) =  4 \, \frac{S + \frac{1}{\ell_1}}{\ell_0S + \frac{\ell_0}{\ell_1} +1}.
	\]
	This follows from the fact that, by Corollary~\ref{Cor: Alternative charakterization of lambda 1} and Proposition~\ref{Prop: Harmonic extension}, the function $\widehat f$ is harmonic on $\widehat G\setminus \{x,y\}$ with $f(x) = 1$, $f(y) =-1$, and thus takes the values
	\[
		f(x) = 1, \qquad f(y) = - 1, \qquad  f(z) = \frac{\ell_0 S+ \frac{\ell_0}{\ell_1} - 1}{\ell_0 S + \frac{\ell_0}{\ell_1} +1},
	\]
	in the vertices $x,y$ and $z$ of $\widehat G$.

	Using the assumption that $\deg(x) \ge 3$, we then construct a function $\widehat g \in H^1(\widehat G)$ with smaller energy than $\widehat f$. Note first that
	\begin{equation} \label{eq:DefinitionMysteryPoint}
		\ell_1S = \frac{\ell_1}{\ell_2} + \dots + \frac{\ell_1}{\ell_{\deg(x)}} \geq \deg(x) -1 \geq 2.
	\end{equation}
	This permits us to define a point in the interior of the edge $e_1=[0, \ell_1]$ emanating at $x$ as $x' := \frac{\ell_1S  -1}{2S}$.
	Let $\widehat g := f_{x', y} \in H^1(\widehat G)$ be the harmonic function on $\widehat G \setminus \{x',y\}$ defined in \eqref{eq:Definitionfxy}. We use Proposition \ref{prop: surgery} to calculate the energy of $\widehat{g}$. First collapse the edges $e_2, \dots, e_{\deg x}$ to a new edge of length $S^{-1}$. Then the new graph is a $[2,1]$-pumpkin chain between the vertices $x', z$ and $y$. By again collapsing the $2$-pumpkin, we get a new edge of length $\frac{\ell_1S+1}{4S}$. The resulting graph is a path graph of length $\frac{\ell_1S+1}{4S}+\ell_0$ and thus $\widehat g$ has energy
	\[
		q(\widehat g) = \frac{16S}{(4\ell_0 + \ell_1)S +1}.
	\]
	Using that $\ell_1 S \neq 1$ by~\eqref{eq:DefinitionMysteryPoint}, it is straightforward to deduce that
	\[
		q(\widehat g) <  q(\widehat f) =  \lambda_{1}^{\min}(\widehat{G}),
	\]
	which is a contradiction by Corollary~\ref{Cor: Alternative charakterization of lambda 1}.
\end{proof}

\begin{rmk}
	The idea of reducing a general metric graph to a pumpkin chain via glueing appeared also in~\cite{surgery}, where it was used in context with eigenvalue estimates (see, e.g.~\cite[Proof of Lemma 5.1]{surgery}). Proof techniques involving changes to the geometry of a metric graph are sometimes also called graph surgery methods.
\end{rmk}
We omit the proof of the following elementary fact.

\begin{lem} \label{lem:DecreasingEdgeLength}
	Let $G =(V,E, \ell)$ be a metric graph. Shortening an edge $e \in E$, that is, decreasing the length $\ell(e)$, increases the first optimal eigenvalue $\lambda_1^{\min}$. Moreover, contracting an edge $e \in E$, that is, collapsing it into one vertex, also increases $\lambda_1^{\min}$.
\end{lem}

\subsection{Estimates via diameter}

Theorem~\ref{Lem: Resitance metric and optimal eigenvalue} characterizes the optimal eigenvalue $\lambda_1^{\min}(G)$ in terms of the resistance metric $r$. If one considers instead the path metric $\varrho$ and its diameter $\diam(G) =\diam_\varrho(G)$, the following estimate holds.
\begin{cor}\label{Cor: First bounds for lambda1}
	Let $G$ be a metric graph. Then
	\begin{equation} \label{eq:Lambda1PathMetricDiameter}
		\frac{4L(G)}{\diam(G)^2} \geq \lambda_1^{\min}(G) \geq \frac{4}{ \diam(G)}.
	\end{equation}
\end{cor}
\begin{proof}
	This follows from Theorem~\ref{Lem: Resitance metric and optimal eigenvalue} and the inequalities for the metrics $\varrho$ and $r$ in~\eqref{eq:InequalityResistancePathMetric} and Lemma~\ref{lem:EasyUpperEstimate}.
\end{proof}
\begin{rmk}
	By the above, $\lambda_1(H_\mu) \ge \frac{4}{\diam(G)}$ for every measure $\mu \in \cN$. Taking $\mu$ with $\supp \mu=V$, we recover the estimate for weighted discrete Laplacians from~\cite[Corollary 3.7]{lss} (see also Example~\ref{ex:DiscreteLaplacians}).
\end{rmk}

In particular, $\lambda_1^{\min}(G)$ can be controlled by the diameter and length of a graph. It turns out that upper bounds in terms of only one of these quantities are impossible.

Denote by $\mathscr{G}_L$ the class of metric graphs $G$ having total length $L(G) = L$. Similarly, let $\mathscr{G}'_L$ be the class of metric graphs $G$ with diameter $\diam(G) = L$.

\begin{cor}
	Let $L >0$. Then
	\[
		\sup_{G \in \mathscr{G}_L} \lambda_1^{\min}(G) = \sup_{G \in \mathscr{G}'_L} \lambda_1^{\min}(G) = + \infty.
	\]
\end{cor}
\begin{proof}
	Concerning $\mathscr{G}_L$, consider the equilateral $n$-star consisting of $n$ edges of length $L/n$ glued at a central vertex. The diameter in the path metric $\varrho$ is $2L/n$. Thus, by Corollary~\ref{Cor: First bounds for lambda1}, the first optimal eigenvalue tends to $+ \infty$ for $n \to \infty$.

	In $\mathscr{G}_L'$, consider the equilateral $[n\times  n]$-pumpkin chain $W_n$, $n \in \N$ (that is, consisting of $n$ pumpkins with $n$ edges each), with constant length function $\ell \equiv \frac{L}{n}$ (see Definition \ref{def: pumpkins}). Then $\diam (W_n) = L$ and $L(W_n) = L n$. Let $x,y \in G$ be points realizing the $r$-diameter $\diam_r(W_n) = r(x,y)$. We can assume that each of the points $x$ and $y$ belongs either to the first or last pumpkin. Due to Proposition \ref{prop: surgery} one can replace each of the $(n-2)$ pumpkins in the middle by a path of length $\frac{L}{n^2}$, without changing the resistance $r(x,y)$ of $x$ and $y$. The new graph $W_n'$ has diameter $\diam(W'_n) =\frac{2L}{n} + \frac{(n-2)L}{n^2}$, so
	\[
		\lambda_{1}^{\min}(W_n) \geq \lambda_{1}^{\min}(W_n') \geq  \frac{4}{\frac{2L}{n} + \frac{(n-2)L}{n^2}} \to \infty, \qquad \text{ for } n \to \infty,
	\]
	by Theorem~\ref{Lem: Resitance metric and optimal eigenvalue} and Corollary~\ref{Cor: First bounds for lambda1}.
\end{proof}

The lower bound in~\eqref{eq:Lambda1PathMetricDiameter} is sharp, as equality holds for the path $P_L =[0,L]$. It turns out that $P_L$ also minimizes $\lambda_1^{\min}$ on some classes of metric graphs.

\begin{cor}
	The path graph $P_L = [0,L]$ minimizes the first optimal eigenvalue in both the classes $\mathscr{G}_L$ and $\mathscr{G}'_L$, that is,
	\[
		\min_{G \in \mathscr{G}_L} \lambda_1^{\min} (G) = \min_{G \in \mathscr{G}'_L} \lambda_1^{\min} (G) = \lambda_1^{\min} (P_L) = \frac{4}{L}.
	\]
	Moreover, $P_L$ is the unique minimizer in $\mathscr{G}_L$, whereas other minimizers exist in $\mathscr{G}'_L$.
\end{cor}
\begin{proof}
	Since $P_L$ is a tree, $\lambda_1^{\min}(P_L) = 4/\diam(P_L) = 4/L$ by Theorem~\ref{Lem: Resitance metric and optimal eigenvalue}. Since $\diam(G) < L$ for every other graph $G \in \mathscr{G}_L$, the path graph $P_L$ is the unique minimizer  in $\mathscr{G}_L$. On the other hand, $\lambda_1^{\min}(T) = 4/L$ holds for every metric tree of diameter $L$, so all of them minimize $\lambda_1^{\min}$ in $\mathscr{G}_L'$.
\end{proof}

\subsection{An upper estimate} Theorem~\ref{thm:SpectralPartitioning} indicates that the first optimal eigenvalue is related to graph partitioning. In the following, we elaborate on this connection by proving an upper estimate for $\lambda_1^{\min}(G)$ in terms of a Cheeger-type constant.

Let $G= (V,E,\ell)$ be a metric graph. For a set of vertices $Y \subset V$, we define its \emph{combinatorial boundary} as
\[
	\delta Y = \{e \in E \; | \; e \text{ connects a vertex in } Y \text{ and in }V\setminus Y\}
\]
and the \emph{volume of the combinatorial boundary} as
\[
	| \delta Y | := \sum_{e \in \delta Y} \frac{1}{\ell(e)}.
\]
This leads to the following upper bound for the first optimal eigenvalue.
\begin{lem} \label{lem:BasicCheeger}
	Consider a subset $Y \subset V$ with $\emptyset\subsetneq  Y \subsetneq V$. Then
	\[
		\lambda_{1}^{\min}(G) \leq 4|\delta Y|.
	\]
\end{lem}
\begin{proof}
	Consider the edgewise linear function $f \in H^1(G)$ which satisfies  $f \equiv 1$ on $Y$ and $f \equiv -1$ on $V \setminus W$. Then $q(f) = 4|\delta Y|$. Applying Corollary~\ref{Cor: Alternative charakterization of lambda 1}, we arrive at the claimed estimate.
\end{proof}

We introduce a \emph{Cheeger-type constant} of a metric graph $G$ as
\[
	h(G) := \min \Big \{ \min_{\emptyset \subsetneq  Y \subsetneq V} | \delta Y|, \, \min_{e_1\neq e_2 \in E} \frac{1}{\ell (e_1)} + \frac{1}{\ell(e_2)}, \, \min_{e \in E} \frac{4}{\ell(e)} \Big \}.
\]
Note that this expression depends on the choice of the model of $G$. To be more precise, we define $h(G)$ as the above expression for one (equivalently, any) essential model of $G$ (see Section~\ref{ss:MetricGraphsBasics}). In considerations involving $h(G)$, we will always work with an essential model of $G$. Since removing vertices of degree $2$ from a model decreases the above expression, this corresponds to taking the minimal constant among all models.

\begin{rmk}
	Comparing $h(G)$ to established notions of Cheeger constants for graphs, its shape might be a bit surprising. Indeed, the conceptual idea behind Cheeger constants is to partition a graph into two subsets of the same volume and with smallest possible boundary. However, our definition of $h(G)$ seemingly does not take into account the volume of subsets. This is due to the shape of the minimizing measure $\mu = \mu_{x,y} = (\delta_x + \delta_y)/2$, which gives the same volume $1/2$ to any subset containing only one of the points $x$ and $y$.
\end{rmk}

It turns out that $h(G)$ has the following simpler representation.
\begin{lem}
	Let $G$ be a metric graph. Then
	\[
		h(G) = \min \Big \{\min\limits_{\substack{e\in E \\ e \text{ bridge}}} \frac{1}{\ell(e)},\min_{e_1\neq e_2 \in E} \frac{1}{\ell (e_1)} + \frac{1}{\ell(e_2)}, \min_{e \in E} \frac{4}{\ell(e)} \Big \},
	\]
	where a bridge of the metric graph $G$ is an edge $e$ such that $G \setminus e$ is disconnected.
\end{lem}
\begin{proof}
	Assume that $Y\subset V$ satisfies $h(G) = |\delta Y|$ and $|\delta Y| < \min_{e \neq f \in E} \frac{1}{\ell (e)} + \frac{1}{\ell(f)}$. Then $\delta Y$ only consists of exactly one edge $e$. If we remove $e$ from $G$, then there is no connection between vertices in $Y$ and $V\setminus Y$, thus $e$ is a bridge.
\end{proof}
As the main result of this section, we estimate $\lambda_1^{\min}(G)$ in terms of $h(G)$.

\begin{prop}\label{Prop: Cheeger}
	Let $G$ be a metric graph. Then
	\[
		\lambda_{1}^{\min}(G) \leq 4 h(G).
	\]
\end{prop}
\begin{proof}
	Given an edge $e \in E$, denote by $\varphi_e$ the unique $H^1$-function with $\varphi_e \equiv 0$ on $G \setminus e$, with $\varphi_e(x) = 1$ for the midpoint $x$ of $e$, and which is linear on both components of $e \setminus \{x\}$. A direct calculation shows that $q(\varphi_e) = 4/\ell(e)$.
	Given two edges $e \neq f$, it follows from Corollary~\ref{Cor: Alternative charakterization of lambda 1} that
	\[
		\lambda_{1}^{\min}(G) \leq q ( \varphi_e - \varphi_f ) = q ( \varphi_e) + q(\varphi_f )  = \frac{4}{\ell(e)} + \frac{4}{\ell(f)}.
	\]
	Breaking an edge $e \in E$ into two edges at the midpoint of $e$ and applying the above argument, we infer that $\lambda_{1}^{\min}(G) \le 16/\ell(e)$ for all edges $e$. Comparing with Lemma~\ref{lem:BasicCheeger}, we conclude.
\end{proof}
The estimate in Proposition~\ref{Prop: Cheeger} is sharp, as equality holds for the path graph $P = [0,1]$. Additional examples for  the sharpness are provided in the next section.

\subsection{Examples}
In this section, we calculate the first optimal eigenvalue $\lambda_1^{\min}$ for several graphs. For metric trees, $\lambda_1^{\min}$ was described in  Theorem~\ref{Lem: Resitance metric and optimal eigenvalue}. Motivated by this, in the following we focus on graphs containing many cycles.
Our key tool is the equality $\lambda_1^{\min} (G) = \min_{x \neq y \in G} q(f_{x,y})$ from~\eqref{eq:ExpressLamda1Fxy}.

\begin{ex}[Complete graph] \label{ex:CompleteGraph}
	Let $K_n$, $n \ge4 $, be the equilateral complete graph with $n$ vertices and constant edge lengths $\ell(e) := \ell := \frac{2}{n(n-1)}$. In particular, $K_n$ has total length $L(K_n) =1$. By~\eqref{eq:ExpressLamda1Fxy}, we can calculate $\lambda_1^{\min}(K_n)$ by minimizing the energy of the function $f_{x,y}$ over all $x,y \in G$. By Theorem~\ref{thm: points avoid vertices}, it suffices to consider points $x$ and $y$ which lie in the interior of edges.

	Assume that $x \in e_1$ and $y \in e_2$ for two disjoint edges $e_1=uv$ and $e_2=wz$ (i.e., $\{u,v\}\cap \{w,z\}  =\varnothing$). Identify $x$ with the point $0 < a < \ell$ in the interval $e_1 = [0, \ell]$ and $y$ with the point $0 < b < \ell$ in the interval $e_1 = [0, \ell]$. By symmetry, the function $f= f_{x,y}$ takes the same value on all vertices $\hat v \in  V\setminus  \{u,v,w,z\}$. Fixing a vertex $\hat v \in V\setminus  \{u,v,w,z\}$, the harmonicity conditions in the vertices $u,v,w,z$ and $\hat v$ can be written as
	\[
		\begin{cases}
			\ell f(x) + a \big (f(w) +   f(z)  + (n-4)f(\hat v) \big ) = (\ell + (n-2)a )f(u)             \\[2mm]
			\ell f(x) + (\ell-a)\big ( f(w) + f(z) + (n-4) f(\hat v) \big )= (\ell + (n-2)(\ell -a) )f(v) \\[2mm]
			\ell f(y) + b \big (f(u) +   f(v)  + (n-4)f(\hat v) \big ) = (\ell + (n-2)b )f(w)             \\[2mm]

			\ell f(y) + (\ell-b)\big ( f(u) + f(v) + (n-4) f(\hat v) \big )= (\ell + (n-2)(\ell -b) )f(z) \\[2mm]

			f(u) + f(v) + f(w) + f(z) = 4 f(\hat v).
		\end{cases}
	\]
	These equations allow to determine the function $f$. Minimizing the corresponding energy and applying~\eqref{eq:ExpressLamda1Fxy}, we calculate
	\[
		\min_{x \in e_1, \, y \in e_2} q(f_{x,y}) = \frac{4n^2\left(n-1\right)}{n+2},
	\]
	where the minimum is achieved for $x$ and $y$ being the edge midpoints. In a similar way, one can show that, if one places points $x$ and $y$ on the same edge or on two adjacent edges, then the corresponding energy is bigger. Thus,
	\[
		\lambda_{1}^{\min}(K_n)   = \frac{4n^2\left(n-1\right)}{n+2}
	\]
	and the minimizing measures are the measures $\mu_{x,y}$ where $x$ and $y$ are the midpoints of disjoint edges. Note that in this case also $\varrho(x,y) = \frac{4}{(n-1)n} = \diam_\varrho K_n$.

	The Cheeger-type constant of $K_n$ is given by
	\[
		h(K_n) = n(n-1).
	\]
	The estimate in Proposition~\ref{Prop: Cheeger} is asymptotically sharp for $K_n$ as $n \to \infty$.
\end{ex}
\begin{ex}[Pumpkin graph] \label{ex: Pumpkin}
	Let $G_n$, $n \ge  2$, be the equilateral $n$-pumpkin, that is, the metric graph consisting of two vertices $x_0$ and $x_1$, joined by $n$ edges of the same length $\ell := 1/n$. Proceeding as in Example~\ref{ex:CompleteGraph}, consider two points $x,y \in G  \setminus V$. We distinguish two cases:

	Assume first that $x$ and $y$ belong to the same edge $e$. Set $a := \varrho(x_0,x)$ and $b := \varrho(x_1, y)$, so that $\varrho(x,y) = 1/n -a-b$.
	As is readily verified, the harmonic function $f=f_{x,y}$ takes the values
	\[
		f(x_0) = \frac{n(n-1)(b-a) + 1 }{n(n-1)(a+b) + 1}, \qquad f(x_1) = \frac{n(n-1)(b-a) - 1}{n(n-1)(a+b) + 1},
	\]
	in the vertices $x_0$ and $x_1$. Its energy is given by
	\[
		q(f) = \frac{4n^2}{ \big (n(n-1)(a+b) + 1 \big ) \big (1-n (a+b) \big )},
	\]
	which becomes minimal precisely when $a+b=\frac{n-2}{2n(n-1)}$. In this case, it equals
	\[
		q(f) = 16(n-1).
	\]

	Now assume that $x$ and $y$ belong to different edges $e_x \neq  e_y$, respectively. Set $a:= \varrho(x_0 ,x)$ and $b := \varrho(x_0,y)$. The harmonic function $f = f_{x,y}$ is then given by the function values
	\[
		f(x_0) = \frac{(a-b)(1-a(n-1)-b(n-1))}{a^2(n-1) + b^2(n-1) + 2ab - a - b}
	\]
	and
	\[
		f(x_1) = \frac{(a-b)(-2-n^2(a+b)+n(a+b+1))}{n(a^2(n-1) + b^2(n-1) + 2ab - a - b)}.
	\]
	Its energy equals
	\[
		q(f) = \frac{4}{a^2(1-n) +a +b^2(1-n) +b -2ab },
	\]
	which becomes minimal for $a=b=1/(2n)$. In this case, we obtain $q(f) = 8n$.

	We conclude that
	\[
		\lambda_1^{\min}(G_n) = 8n
	\]
	and that the minimizing measures are the measure $\mu_{x,y}$, where $x$ and $y$ are midpoints of two distinct edges. Note that in this case $\diam_\varrho(G) = \varrho(x,y)$. Moreover, the estimate from Proposition~\ref{Prop: Cheeger} in terms of $h(G_n) = 2n$ is sharp.
\end{ex}

\begin{ex}[Butterfly graph, a.k.a. ${[}n,  n{]}$-pumpkin chain]
		For $n \ge 2$, consider the butterfly graph $B_n$,  that is the graph consisting  of three vertices $x_0, x_1$ and $c$, and having $n$ edges each from $c$ to $x_0$ and from $c$ to $x_1$, all of the same length $\ell = \frac{1}{2n}$. Using similar arguments as in the previous two examples, one can calculate that
		\[
			\lambda_1^{\min} (B_n) =16(n-1).
		\]
		The minimizing measures are precisely the measures $\mu_{x,y}$ where $x$ lies on an edge at $x_0$, and $y$ on an edge at $x_1$, with positions $\varrho(x_0,x) = \varrho(x_1, y) =   \frac{n-2}{4n(n-1)}$. Note that $\varrho(x,y) = \frac{1}{2(n-1)} < \diam(B_n) = \frac{1}{n}$ for $n \ge 3$.

		The optimal Cheeger constant is given by $h(B_n) = 4n$, so that Proposition~\ref{Prop: Cheeger} is asymptotically sharp for $n\to\infty$.
\end{ex}


\section{Weyl's law for optimal eigenvalues} \label{sec:Weyl'sLaw}
In this section, we calculate the optimal eigenvalues of the path graph $P =[0,1]$ and derive a Weyl's law for optimal eigenvalues on general metric graphs.

\subsection{Optimal eigenvalues of the path graph}
In the following, we view the interval $[0,1]$ as a metric graph $P$ consisting of two vertices joined by an edge of length one, and call $P$ the path graph (of length one).

\begin{thm} \label{thm:ConfEVPath}
	The optimal eigenvalues of the path graph $P =[0,1]$ are
	\[
		\lambda_k^{\min} ( P ) = 4 k^2, \qquad k \in \N_{0}.
	\]
	Moreover, the unique measure $\mu_k \in \cN$ with $\lambda_k^{\min} ( P ) = \lambda_k \big( H_{\mu_k}\big )$ is
	\begin{equation} \label{eq:MinimizingMeasurePath}
		\mu_k =\frac{1}{2k} \delta_0 + \frac{1}{2k} \delta_1 + \sum_{j=1}^{k-1} \frac{1}{k} \delta_{\frac{j}{k}} \in \cN.
	\end{equation}
\end{thm}
In the proof of Theorem~\ref{thm:ConfEVPath}, we need the following lemma.

\begin{lem} \label{lem:SpectralGapEstimatesDirichletLaplacians}
	Consider the path graph $P_L = [0,L]$ of length $L>0$. For the Dirichlet sets $A =\{L\}$ and $A' =\{0, L\}$, the smallest optimal eigenvalues are given by
	\[
		\lambda_0^{\min, D}\big (P_L, A \big ) = \frac{1}{L}, \qquad \lambda_0^{\min, D}\big (P_L, A'\big ) = \frac{4}{L}.
	\]
	The unique minimizing measures are given by $\mu = \delta_0$ and $\mu' = \delta_{L/2}$, respectively.
\end{lem}

\begin{proof}
	To prove the first claim, note that every minimizing measure for $\lambda_0^{\min, D}$ is a Dirac measure due to Lemma \ref{lem:MinimizingMeasuresDirichletEigenvalue}. Since $\lambda_0(H_{\delta_x}^D) = 1/(L-x)$ for every $x \in [0,L)$, the claim follows. Similarly, to prove the second claim, we need to minimize $\lambda_{0}(H^D_{\delta_x})$ over all measures $\delta_x$ for $x \in (0,L)$. A simple calculation shows that $\lambda_{0}(H^D_{\delta_x}) = \frac{1}{x} + \frac{1}{L-x}$, which becomes minimal precisely for $x = \frac{L}{2}$.
\end{proof}

\begin{proof}[Proof of Theorem~\ref{thm:ConfEVPath}]
	The statement is trivial for $k=0$. For $k \in \N$, we consider the vertex set $V = \{\frac{j}{k} | \, j=0, \dots, k\}$ and join the vertices $\frac{j}{k}$ and $\frac{j+1}{k}$ by an edge of length $\frac{1}{k}$ for all $j$, which gives a model of $P$. By Example~\ref{ex:DiscreteLaplacians}, the operator $H_{\mu_k}$ for the measure $\mu_k$ from~\eqref{eq:MinimizingMeasurePath} is a discrete Laplacian. In fact, $H_{\mu_k} = 2k^2 H_{\rm{norm}}$, where $H_{\rm{norm}} \colon \R^V \to \R^V$ is the normalized Laplacian acting on functions $f \in \R^V$ as
	\[
		(H_{\rm{norm}} f)(v) =\frac{1}{\deg(v)} \sum_{w \sim v} (f(v)-f(w)), \qquad v \in V.
	\]
	It is well-known that $\lambda_k(H_{\rm{norm}})= 2$ (e.g., \cite[Example 1.4]{chung}) and we conclude that
	\[
		\lambda_k^{\min}(P) \le \lambda_k(H_{\mu_k}) = 4k^2.
	\]

	It remains to prove that $\lambda_k^{\min}(H_\mu)  \ge 4k^2$. Fix a minimizing measure $\mu \in \cN$. By the Sturm oscillation theorem~\ref{Thm: Sturm oscillation}, the $\mu$-normalized eigenfunction $f$ for  $\lambda_k(H_\mu)$ has precisely $k$ zeros $0<x_1<\dots < x_k <1$. Set $x_0 :=0$, $x_{k+1} :=1$ and consider the intervals $I_i:=(x_i, x_{i+1})$, $i= 1 , \dots, k$. Then
	\[
		\frac{\int_{I_i} |f'|^2 dx }{\int_{I_{i}} |f|^2d\mu} = \frac{\int_{P} f' \cdot (f|_{I_i})' dx }{\int_{I_{i}} |f|^2d\mu} = \frac{\int_{P} H_{\mu}f \cdot f|_{I_i} d\mu }{\int_{I_{i}} |f|^2d\mu} = \lambda_k(H_\mu) \frac{\int_{I_i} |f|^2 d\mu }{\int_{I_{i}} |f|^2d\mu} = \lambda_k(H_\mu),
	\] where $f|_{I_i}$ is extended by zero to $P$. Lemma \ref{lem:SpectralGapEstimatesDirichletLaplacians} implies that
	\[
		4k^2 \ge \lambda_k(H_\mu) = \frac{\int_{I_i} |f'|^2 dx }{\int_{I_{i}} |f|^2d\mu} \geq \max_i \left(\frac{1}{\mu(I_0)|I_0|},\frac{1}{\mu(I_k)|I_k|}  ,\frac{4}{\mu(I_i)|I_i|} \right),
	\]
	where $I_i := (x_{i+1}-x_i)$. Define $s_i = \mu(I_i)+|I_i|$, $i=0, \dots, k$. Then
	\[
		s_i^2 \ge 4 \mu(I_i) |I_i|.
	\]
	Moreover, one of the inequalities
	\[
		s_0 > \frac{1}{k}, s_k > \frac{1}{k} \text{ and } s_i > \frac{2}{k} \quad \text{ for all } i=1, \dots, k-1,
	\]
	must fail. Indeed, otherwise we have the contradiction
	\[
		2  \ge \sum_{i=0}^{k} s_k > \frac{1}{k} + \frac{1}{k} + (k-1) \frac{2}{k} = 2.
	\]
	If $s_0 \leq \frac{1}{k}$ or $s_k \leq \frac{1}{k}$, then
	\[
		\mu(I_i)|I_i| \leq \frac{s_i^2}{4} \leq \frac{1}{4 k^2}, \quad \text{ for } i=0 \text{ or } i=k.
	\]
	If $s_i \leq \frac{2}{k}$ for some $i=1, \dots, k-1$, then
	\[
		\mu(I_i)|I_i|	 \leq \frac{s_i^2}{4} \leq \frac{1}{k^2}.
	\]
	In both cases, we obtain that
	\[
		4k^2 \ge \lambda_k(H_\mu) \ge \max_i \left(\frac{1}{\mu(I_0)|I_0|},\frac{1}{\mu(I_k)|I_k|}  ,\frac{4}{\mu(I_i)|I_i|} \right) \geq 4k^2
	\] and thus $\lambda_k^{\min}(P)=4k^2$. Moreover, by the above estimates, this implies that
	\[
		s_0 = \frac{1}{k},\; s_k = \frac{1}{k} \text{ and } s_i = \frac{2}{k} \quad \text{ for all } i=1, \dots, k-1,
	\] and
	\[
		\mu(I_i) = |I_i| = s_i /2, \qquad \text{ for all } i.
	\]
	By Lemma \ref{lem:SpectralGapEstimatesDirichletLaplacians}, the unique minimizing measure on each interval $I_i$ is a Dirac measure. Combining the above, we get that $\mu$ coincides with $\mu_k$ defined in~\eqref{eq:MinimizingMeasurePath}.
\end{proof}

\subsection{Weyl's law}
In the following, we derive a Weyl's law for the optimal eigenvalues of a metric graph. The result is based on the following lemma.

\begin{lem} \label{lem:EigenvalueInterlacingConsequence}
	Let $G=(V,E, \ell)$ be a metric graph with optimal eigenvalues $\lambda_k^{\min}(G)$ for $k \in \N_{0}$. Then
	\begin{equation*} \label{eq:EigenvalueInterlacingConsequence1}
		\lambda_k^{\min}(G) \ge 	\frac{4 (k-(\#E - 1))^2}{L(G)}
	\end{equation*}
	for all $k \ge \#E-1$  and
	\begin{equation*} \label{eq:EigenvalueInterlacingConsequence2}
		\lambda_k^{\min}(G) \le \frac{4 \big(k + 2\#E - \#V  \big)^2}{L(G)}
	\end{equation*}
	for all $k \ge 0$.
\end{lem}
\begin{proof}
	Our strategy is to compare $G$ to a path graph by first ``cutting'' $G$ into intervals and then ''glueing'' them again to a path graph $P_L = [0, L] = [0,L(G)]$.

	More precisely, we cut the graph $G$ into intervals $[0, \ell(e)]$, $e\in E$, and consider their disjoint union $\widetilde G := \bigcup_{e \in E} [0, \ell(e)]$. Let $\mu \in \cN$ be a measure on $G$ of the form $\mu = m dx$ for some function $m >0$. Restricting $\mu$ to each interval edge $e \in E$, we obtain a measure $\widetilde \mu$ on $\widetilde G$.

	We compare the associated Laplacians $H_\mu$ and $H_{\tilde \mu}$. By the properties of $\mu$, one can naturally identify $L^2(G, \mu) \cong L^2(\widetilde G, \widetilde\mu)$. Moreover, the operators have form domains $H^1(G, \mu) = H^1(G)$ and $H^1(\widetilde G, \widetilde \mu) = H^1(\widetilde G) = \bigoplus_{e \in E} H^1([0, \ell(e)])$, respectively. We can actually view $H^1(G)$ as a subspace of $H^1(\widetilde G)$, consisting of those functions $f \in H^1(\widetilde G)$ with $f(a) = f(b)$, whenever $a$ and $b$ are endpoints of intervals which corresponded to the same vertex in $G$. Since every vertex $v \in V$ gives rise to $\deg(v) -1$ such conditions, the co-dimension of $H^1(G)$ in $H^1(\widetilde G)$ is precisely $\sum_{v \in V} \deg(v)-1 = 2 \#E - \#V$. The quadratic forms of $H_\mu$ and $H_{\tilde \mu}$ coincide on $H^1(G) \subset H^1(\widetilde G)$.  Thus, the standard interlacing principle yields that
	\[
		\lambda_k (  H_{\widetilde \mu} )  \le \lambda_k (H_\mu) \le \lambda_{k + 2\#E - \#V} (H_{\widetilde \mu}), \qquad k \in \N_{0}.
	\]
	Minimizing over all measures $\mu$ and applying a density argument, we obtain
	\begin{equation} \label{Interlacing1}
		\lambda_k^{\min}(\widetilde G)  \le \lambda_k^{\min}(G) \le \lambda_{k + 2\#E - \#V}^{\min}(\widetilde G), \qquad k \in \N_0.
	\end{equation}
	On the other hand, we can glue the intervals in $\widetilde G$ to a path graph $P_L = [0,L]$ with $\#E$ edges.  Applied to $P_L$, the inequalities~\eqref{Interlacing1} read as
	\begin{equation} \label{Interlacing2}
		\lambda_k^{\min}(\widetilde G)  \le \lambda_k^{\min}( P_L) \le \lambda_{k + \#E -1}^{\min}(\widetilde G), \qquad k \in \N_0.
	\end{equation}
	The claimed estimates then follow from Theorem~\ref{thm:ConfEVPath} (see also Remark\ref{rmk : Scaling property}).
\end{proof}

\begin{thm}[Weyl's law] \label{thm:Weyl'sLaw}
	Let $G$ be a metric graph and let $N^{\min} \colon [0, \infty) \to [0, \infty)$ be the counting function for optimal eigenvalues, defined by
	\[
		N^{\min}(\lambda) := \#  \big  \{k \in \N_0|\, \lambda_k^{\min}(G) \le \lambda \big \}.
	\]
	Then
	\begin{equation} \label{eq:Weyl2}
		\frac{N^{\min}(\lambda)}{\sqrt{\lambda}}  \sim \frac{\sqrt{L(G) }}{2}, \qquad \lambda \to \infty,
	\end{equation}
	and
	\begin{equation} \label{eq:Weyl1}
		\lambda_k^{\min}(G) \sim \frac{4 k^2}{L(G)}, \qquad k \to \infty.
	\end{equation}
\end{thm}
\begin{proof}
	These asymptotics follow immediately from Lemma~\ref{lem:EigenvalueInterlacingConsequence}.
\end{proof}


\appendix

\section{A Sturm oscillation theorem} \label{sec:SturmOscillationTheorem}
We prove the following Sturm oscillation theorem on the path graph $P = [0,1]$.

\begin{thm}[Sturm oscillation theorem for general measures] \label{Thm: Sturm oscillation}
	Let $\mu \in \cM$ be a measure on the path graph $P=[0,1]$ and let $f$ be an eigenfunction of $H_\mu$ for the eigenvalue $\lambda_k(H_\mu)$. Then $f$ has precisely $k$ zeros and all of them lie in $(0,1)$.
\end{thm}
Since $P$ coincides with the interval $P=[0,1]$, the operator $H_\mu$ is a Sturm--Liouville operator with measure-valued coefficients. The study of such operators has a long tradition~\cite{feller, kk}. In particular, different versions of Sturm oscillation and comparison theorems have been obtained in this setting, see e.g.~\cite{App: no controll at the boundary,  App: Dirichlet-case,App: Comparison results,  App: Volkmer}. On the other hand, we were not able to find the above variant, which is needed in the proof of Theorem~\ref{thm:ConfEVPath}, in the literature.

Consider a bounded interval $(a,b)$ and a locally finite Borel measure $\varrho$ on $(a,b)$. Let $\mathrm{AC}_{\mathrm{loc}}((a,b), \varrho)$ be the space of functions $f$ such that there exists a function $h \in L_{\mathrm{loc}}^1((a,b), \varrho )$ with
\[
	f(y) = f(x) + \int_{[x,y)} h \, d\varrho
\]
for all $x<y$. Every function $f \in \mathrm{AC}_{\mathrm{loc}}((a,b), \varrho)$ is left-continuous and we set
\[
	\frac{d}{d \varrho} f := h.
\]
We are interested in the differential expression
\[
	\tau = - \frac{d}{d\varrho} \frac{d}{dx} f,
\]
which is well-defined for $f$ belonging to the domain
\[
	\mathcal{D}_\tau = \big \{f \in \mathrm{AC}_{\mathrm{loc}}((a,b), dx)| \, f' \in \mathrm{AC}_{\mathrm{loc}}((a,b), \varrho) \big \}.
\]
More precisely, the condition $f' \in \mathrm{AC}_{\mathrm{loc}}((a,b), \varrho)$ means that $f'$ has a representative in $ \mathrm{AC}_{\mathrm{loc}}((a,b) ,\varrho)$, which is moreover unique by left-continuity. When referring to $f'$ for some $f \in \mathcal{D}_\tau$, we mean this particular representative.

We need the following result on uniqueness of solutions (e.g.,~\cite[Theorem 3.1]{Sturm-Liouville}).

\begin{lem}\label{lem: Unique continuation for ODE}
	Fix a point $c\in (a,b)$ and let $\alpha, \beta, \lambda \in \R$. Then there exists a unique solution $f \in \mathcal{D}_\tau$ of the equation $\tau f = \lambda f$ satisfying
	\[
		f(c) = \alpha \qquad \text{and} \qquad f'(c) = \beta.
	\]
\end{lem}

Consider the path graph $P  = [0,1]$ and let $\mu \in \cM$. The corresponding Laplacian $H_\mu \colon \dom(H_\mu) \to L^2(G, \mu) $ relates to the differential expression $\tau$ for $\varrho := \mu |_{(0,1)}$. More precisely, we have $\dom(H_\mu) \subset \mathcal{D}_\tau$ and
\[
	\tau f = -\frac{d}{d\mu} \frac{d}{dx} f  = H_\mu f \quad \text{in } L^1((0,1), \mu)
\]
for every $f \in \dom(H_\mu)$. To see this, let $c < x \in(0,1)$ be two points in which $f$ is differentiable. For $n \in \N$, take $\phi_n \in H^1(G)$ as the test function with $\varphi_n \equiv 1$ on $[c, x-1/n]$, with $\varphi_n \equiv 0$ on $[0, c-1/n] \cup [x,1]$ and which is linear everywhere else on $P=[0,1]$. Then
\begin{align*}
	\int_{[c,x)} (H_\mu f) d \mu & = \lim_{n \to \infty} \int_{0}^{1} (H_\mu f) \phi_n d\mu = \lim_{n \to \infty} q(f, \phi_n)  \\
	                             & = \lim_{n \to \infty} n \int_{c-1/n}^{c} f'(s)ds - n\int_{x-1/n}^{x} f'(s)ds= f'(c) - f'(x).
\end{align*}

A similar argument shows that the right and left derivatives in the endpoints of $P= [0,1]$ exist for all $f \in \dom(H_\mu)$ and equal
\begin{equation} \label{eq:BoundaryConditionsPathGraphs}
	f_+'(0) = - \mu(\{0\}) (H_\mu f) (0), \qquad  f_-'(1) = \mu(\{1\}) (H_\mu f) (1).
\end{equation}
Here, the respective right hand side is zero, whenever $\mu(\{0\})$ or $\mu(\{1\})=0$.
\begin{rmk}
	The above generalizes to an arbitrary metric graph $G$ and measure $\mu \in \cM$. Namely, for $f \in \dom(H_\mu)$, on every interval edge $(0, \ell(e))$, $e \in E$, we have
	\[
		H_\mu f = -\frac{d}{d\mu} \frac{d}{dx} f \qquad \text{in $L^2( (0, \ell(e)), \mu)$},
	\]
	where the restriction $f|_{(0,\ell(e))}$ belongs to the respective domain $\mathcal{D}_{\tau_e}$. Moreover,
	\[
		\sum_{e \sim v} f_e'(v) = - \mu(\{v\}) (H_\mu f)(v), \qquad \text{for all $v \in V$},
	\]
	where all edges $e \sim v$ are oriented pointing away from $v$.

	From this perspective, the Laplacian $H_\mu$ may be viewed as a graph variant of a Krein string operator. Note that such operators have already been studied before on star graphs in~\cite{eckhardt, prt}. Under additional assumptions on the measure $\mu$, our operators also coincide with those for general graphs in~\cite{kkv}.
\end{rmk}

\begin{prop} \label{Pro: Simplicity of eigenvalues on the interval}
	Let $\mu \in \cM$. Then all eigenvalues of $H_\mu$ are simple.
\end{prop}
\begin{proof}
	It suffices to prove that every solution $f \in \dom(H_\mu)$ of $H_\mu f = \lambda f$ with $f(0) = 0$ vanishes everywhere on $P=[0,1]$. Taking into account~\eqref{eq:BoundaryConditionsPathGraphs}, such a function $f$ necessarily satisfies $f_+'(0) =0$. Since the equation $H_\mu f = \lambda f$ is linear of second order, the initial conditions $f(0)= f_+'(0)= 0$ should already imply that $f \equiv 0$.
	More formally, this can be deduced from Lemma~\ref{lem: Unique continuation for ODE}. Extend $\mu$ and $f$ by zero to a measure $\nu$ and function $g \in \dom(H_\nu)$ on $J =(-1, 1)$. Then $-\frac{d}{d\nu} \frac{d}{d x } g = \lambda g$ on $J$ and $g(0)= g'(0)= 0$, so it follows that $g \equiv 0$ and $f \equiv 0$.
\end{proof}

\begin{lem} \label{lem: upper zero bound}
	Let $\mu \in \cM$ and let $f \in H^1(P, \mu)$ be an eigenfunction for the $k$-th eigenvalue $\lambda_k(H_\mu)$. Then $f$ has at most $k$ zeros and all its zeros lie in $(0,1)$.
\end{lem}
\begin{proof}
	We already know that $f(0), f(1) \neq 0$ by the proof of Proposition~\ref{Pro: Simplicity of eigenvalues on the interval}. Suppose that $f$ has at least $l>k$ zeros $0 < x_1 < x_2 < \dots < x_{l} < 1$. Set also $x_0:=0$  and $x_{l+1} :=1$. For $i=0, \dots, l$, let $f_i \in H^1(G)$ be the function obtained by restricting $f$ to $[x_i, x_{i+1}]$ and extending to $[0,1]$ by zero. These $l+1$ functions are all non-trivial, by Lemma~\ref{lem: Unique continuation for ODE}, and $\mu$-orthogonal by construction. The same holds for their projections $\mathcal{P}_\mu(f_i)$, $i=0, \dots, l$. By the Min-Max-principle and Proposition~\ref{Prop: Harmonic extension},
	\[
		\lambda_l \leq  \max_{f \in \Span\{f_0, \dots, f_l\}} \frac{ q(\mathcal{P}_\mu(f)))}{\|\mathcal{P}_\mu(f)\|^2_{L^2(G, \mu)}} \le \max_i \frac{ q(f_i)}{\|f_i\|^2_{L^2(G, \mu)}} =  \max_{i} \frac{\lambda_k\|f_i\|^2_{L^2(G, \mu)}}{\|f_i\|^2_{L^2(G, \mu)}} = \lambda_k.\]
	We obtain that $\lambda_k = \lambda_l$, in contradiction to Proposition~\ref{Pro: Simplicity of eigenvalues on the interval}.
\end{proof}
If $\mu \in \cM$ has a positive $L^1$-density, i.e.~$\mu = r dx$ for $r \in L^1(0,1)$ with $r>0$, then
\[
	\tau = - \frac{1}{r} \frac{d^2}{dx^2}.
\]
Moreover, the Laplacian $H_\mu$ coincides with the standard self-adjoint Sturm--Liouville operator in $L^2((0,1), \mu)$ associated with the Neumann boundary conditions
\[
	f'(0) = f'(1)= 0.
\]
We have the following classical version of the Sturm oscillation and comparison theorem (e.g.,~\cite[Theorem 6.1.1 and Theorem 6.1.2]{Bennewitz2020-hu}).

\begin{prop} \label{Prop: Sturm comparison for smooth densities}
	Consider the Laplacian $H_\mu$ for $\mu \in \cM$ with positive $L^1$-density.

	\begin{itemize}
		\item [(i)] Let $f$ be an eigenfunction for the $k$-th eigenvalue $\lambda_k(H_\mu)$. Then $f$ has precisely $k$ zeros $x_1 < x_2 < \dots < x_k$ and all of them lie in $(0,1)$.
		\item [(ii)] Additionally, let $g$ be an eigenfunction for an eigenvalue $\lambda_l(H_\mu) > \lambda_k(H_\mu)$. Then $g$ has at least one zero in every interval $(x_i, x_{i+1})$ for $i =0, \dots, k$, where $x_0 := 0$ and $x_{k+1} := 1$ by convention.
	\end{itemize}
\end{prop}
Using the above ingredients, we can prove Theorem~\ref{Thm: Sturm oscillation}.

\begin{proof}[Proof of Theorem~\ref{Thm: Sturm oscillation}] We proceed by induction over $k$. The case $k=0$ is clear. Fix eigenfunctions $f_0, \dots, f_k$ for the first $k+1$ eigenvalues $\lambda_0, \dots, \lambda_k$ of $H_\mu$ and suppose that the claim holds for $f_0, \dots, f_{k-1}$.  Choose a sequence of measures $\mu_n \in \cM$, $n \in \N$, with positive $L^1$-densities converging to $\mu$. Since all eigenvalues are simple by Proposition~\ref{Pro: Simplicity of eigenvalues on the interval}, applying Theorem~\ref{thm: spec conv }, we find eigenfunctions $f_{n,0}, \dots, f_{n,k}$ of $H_{\mu_n}$ which converge uniformly to $f_0, \dots, f_k$ as $n \to \infty$. By Proposition~\ref{Prop: Sturm comparison for smooth densities}, each eigenfunction $f_{n,l}$ has precisely $l$ zeros $x_1(f_{n,l}) < x_2(f_{n,l}) < \dots < x_l(f_{n,l})$. We may suppose that each zero $x_j(f_{n,l})$ converges to a point $ x_j(f_l) := \lim_{n \to \infty} x_j(f_{n,l})$. Since $f_{n, l} \to f_l$ uniformly, each limit $x_k(f_l)$ is a zero of $f_l$ and $x_1(f_l) \le x_2(f_l) \le \dots \le x_l(f_l)$. Moreover, $x_j(f_l) \neq 0, 1$ by Lemma~\ref{lem: upper zero bound}.

	It suffices to prove that the limiting zeros $x_j(f_k)$ for $f_k$ are all different, that is, $x_j(f_k) < x_{j+1}(f_k)$ for all $j=1, \dots, k-1$. The result then follows by Lemma~\ref{lem: upper zero bound}.

	We first show that three consecutive zeros cannot collapse, that is, we cannot have $x_j(f_k) = x_{j+1}(f_k) =  x_{j+2}(f_k)$. In such a case, using the interlacing property in Proposition~\ref{Prop: Sturm comparison for smooth densities} for $f_{n,k}$ and $f_{n, k-1}$, we would have two coinciding limiting zeros $x_i(f_{k-1}) = x_{i+1}(f_{k-1})$ of $f_{k-1}$. Applying induction, this is impossible.

	It remains to outrule that $x_{j-1}(f_k)< x_j(f_k) = x_{j+1}(f_k) < x_{j+2}(f_k)$ for some $j=1, \dots, k$ (with convention $x_0(f_k) := 0$ and $x_{k+1}(f_k) :=1$). Since in this case the nodal domain $(x_j(f_{n,k}), x_{j+1}(f_{n,k}))$ of $f_{n,k}$ has collapsed onto $x_j(f_k) = x_{j+1}(f_k)$, the eigenfunction $f_k$ does not change sign on $I_j := [x_{j-1}(f_k), x_{j+2}(f_k)]$. Supposing that $f_k \ge 0$ on $I_j$, the differential equation $H_\mu f_k = \lambda_k f_k$ implies that $f_k'$ is monotonically decreasing and thus $f_k$ is concave on $I_j$. Using that $f_k(x_j) = 0$, we obtain that then $f_k \equiv 0$ on $I_j$, which implies that $f_k \equiv 0$ on $[0,1]$ by Lemma~\ref{lem: Unique continuation for ODE}. This contradiction finishes the proof.
\end{proof}

\ack{We thank the referees for taking the time to review our work and their valuable comments. We are grateful to Omid Amini, Aleksey Kostenko and Delio Mugnolo for useful discussions and hints with respect to the literature, and to Jonathan Eckhardt for his help concerning references on Sturm oscillation theorems for Krein strings. We would like to thank Daniel Grieser for the enriching exchange on the physical background of the operators considered.}


\begin{thebibliography}{XXX}



	\bibitem{Nodalcountandpartition}
	R.\ Band, G.\ Berkolaiko, H.\ Raz and U.\ Smilansky, {\em The Number of Nodal Domains on Quantum Graphs as a Stability Index of Graph Partitions}, Commun.\ Math.\ Phys.\ \textbf{311}, 815--838 (2012).

	\bibitem{Band}
	R.\ Band and G.\ Lévy, {\em Quantum Graphs which Optimize the Spectral Gap}, Ann.\ Henri Poincar\'e {\bf 18}, no.~10, 3269--3323 (2017).

	\bibitem{Bennewitz2020-hu}
	C. Bennewitz, M. Brown and R. Weikard {\em Spectral and scattering theory for ordinary differential equations}, Springer Nature, Universitext, 2020.


	\bibitem{surgery}
	G. Berkolaiko, J. Kennedy, P. Kurasov and D. Mugnolo, {\em Surgery principles for the spectral analysis of quantum graphs}, Trans.\ Am.\ Math.\ Soc.\ \textbf{372}, 5153--5197 (2019)


	\bibitem{BerkolaikoKuchment}
	G.\ Berkolaiko and P.\ Kuchment, {\em Introduction to Quantum Graphs}, American Mathematical Society, Providence, RI, 2013.


	\bibitem{brezis}
	H.\ Brezis, {\em Functional Analysis, Sobolev Spaces and Partial Differential Equations}, Universitext, Springer, 2011.

	\bibitem{bgp}
	J.\ Brüning, V.\ Geyler, and K.\ Pankrashkin, {\em Spectra of self-adjoint extensions and applications to solvable Schrödinger operators}, Rev.\ Math.\ Phys.\ \textbf{20}, no.~1, 1--70 (2008).


	\bibitem{Cao}
	D.\ H. Cao, {\em Geodesic nets via eigenvalue optimisation}, preprint, \arxiv{2508.16331} (2025).


	\bibitem{cgiko}
	S.\ Chanillo, D.\ Grieser, M.\ Imai, K.\ Kurata and I.\ Ohnishi, {\em Symmetry Breaking and Other Phenomena in the Optimization of Eigenvalues for Composite Membranes}, Commun.\ Math.\ Phys.\ \textbf{214}, 315--337 (2000).

	\bibitem{chung}
	F.\ Chung, {\em Spectral Graph Theory}, CBMS Regional Conference Series in Mathematics {\bf 92}, Providence, RI, 1997.

	\bibitem{ce03}
	B.\ Colbois and A.\ El Soufi, {\em Extremal eigenvalues of the Laplacian in a conformal class of metrics: the ‘conformal spectrum’}, Ann.\ Global Anal.\ Geom.\ \textbf{24}, 337--349 (2003).


	\bibitem{ce19}
	B.\ Colbois and A.\ El Soufi, {\em Spectrum of the Laplacian with weights}, Ann.\ Global Anal.\	Geom.\ \textbf{55}, no.~2, 149--180 (2019).

	\bibitem{ces15}
	B.\ Colbois, A.\ El Soufi and A.\ Savo, {\em Eigenvalues of the Laplacian on a manifold with density}, Commun.\ Anal.\ Geom.\ \textbf{23}, no.~3, 639--670 (2015).

	\bibitem{cp18}
	B.\ Colbois and L.\ Provenzano, {\em Eigenvalues of elliptic operators with density}, Calc.\ Var.\ \textbf{57}, no.~36 (2018).

	\bibitem{cmI}
	S.\ J.\ Cox and J.\ R.\ McLaughlin, {\em Extremal eigenvalue problems for composite membranes I}, Appl.\ Math.\ Optim.\ \textbf{22}, 153--167 (1990).

	\bibitem{cmII}
	S.\ J.\ Cox and J.\ R.\ McLaughlin, {\em Extremal eigenvalue problems for composite membranes II}, Appl.\ Math.\ Optim.\ \textbf{22}, 169--187 (1990).


	\bibitem{dav89}
	E.\ B.\ Davies, {\em Heat Kernels and Spectral Theory}, Cambridge Univ.\ Press, Cambridge (1989).

	\bibitem{eckhardt}
	J.\ Eckhardt, {\em An inverse spectral problem for a star graph of Krein strings}, J.\ reine angew.\ Math.\ \textbf{715}, 189--206 (2016).

	\bibitem{Sturm-Liouville}
	J.\ Eckhardt and G.\ Teschl, {\em Sturm-Liouville operators with measure-valued coefficients}, Journal d'Analyse Mathématique \textbf{120}, no.~1, 151--224 (2013).

	\bibitem{ekmn}
	P.\ Exner, A.\ Kostenko, M.\ Malamud and H.\ Neidhardt, \emph{Spectral theory of infinite quantum graphs}, Ann.\ Henri Poincar\'e {\bf 19}, no.~11, 3457--3510 (2018).

	\bibitem{ep05}
	P.\ Exner and O.\ Post, {\em Convergence of spectra of graph-like thin manifolds}, J.\ Geom.\ Phys.\ \textbf{54}, 77--115 (2005).

	\bibitem{feller}
	W.\ Feller, {\em Generalized second order differential operators and their lateral conditions}, Ill.\ J.\ Math.\ \textbf{1}, 459--504 (1957).

	\bibitem{friedland}
	S.\ Friedland, {\em Extremal eigenvalue problems defined for certain classes of functions}, Arch.\ Rational Mech.\ Anal.\ \textbf{67}, no.~1, 73--81 (1977).


	\bibitem{App: no controll at the boundary}
	G.\ Figueiredo and M.\ Montenegro, {\em Positive and sign-changing solutions for nonlinear equations with rapid growing weights}, Applicable Analysis \textbf{102}, no.~6, 1687--1695 (2023).

	\bibitem{App: Dirichlet-case}
	U.\ Freiberg and J.-U.\ Löbus, {\em Zeros of eigenfunctions of a class of generalized second order differential operators on the Cantor set}, Math.\ Nachr.\ \textbf{265}, 3--14 (2004).


	\bibitem{gkl}
	A.\ Girouard, M.\ Karpukhin and J.\ Lagacé, {\em Continuity of eigenvalues and shape optimisation for Laplace and Steklov problems}, Geom.\ Funct.\ Anal.\ {\bf 31}, 513--561 (2021).

	\bibitem{gk}
	C.\ Gohberg and M.\ G.\ Krein, {\em Theory and applications of Volterra operators in Hilbert space}, Transl.\ Math.\ Monogr.\ \textbf{24}, American Mathematical Society, Providence, R.I., 1970.

	\bibitem{grigoryanbook}
	A.\ Grigor'yan, {\em Introduction to Analysis on Graphs}, Univ.\ Lect.\ Series {\bf 71}, American Mathematical Society, Providence, RI, 2018.


	\bibitem{gny}
	A.\ Grigor’yan, Y.\ Netrusov and S.-T.\ Yau, {\em Eigenvalues of elliptic operators and geometric applications}, Surv.\ Differ.\ Geom.\ IX, 147--217 (2004).

	\bibitem{henrot}
	A.\ Henrot, {\em Extremum Problems for Eigenvalues of Elliptic Operators}, Frontiers in Mathematics, Birkhäuser Verlag, Basel, 2006.

	\bibitem{SpecPart2}
	M.\ Hofmann, J.\ B.\ Kennedy, D.\ Mugnolo and M.\ Plümer, {\em Asymptotics and Estimates for Spectral Minimal Partitions of Metric Graphs}, Integr.\ Equ.\ Oper.\ Theory \textbf{93}, no.~26 (2021).


	\bibitem{jp23}
	P.\ E.\ T.\ Jorgensen and E.\ P.\ J.\ Pearse, {\em Operator Theory and Analysis of Infinite Networks}, World Scientific Publishing, Contemporary Mathematics and Its Applications \textbf{7}, 2023.

	\bibitem{kk}
	I.\ S.\ Kac and M.\ G.\ Krein, {\em On the spectral functions of the string}, Am.\ Math.\ Soc.\ Transl.\ Ser.\ 2, \textbf{103}, 19--102 (1974).

	\bibitem{kkv}
	U.\ Kant, T.\ Klauss, J.\ Voigt and M.\ Weber, {\em Dirichlet forms for singular one-dimensional operators and on graphs}, J.\ Evol.\ Equ.\ \textbf{9}, 637--659 (2009).

	\bibitem{ksI}
	M.\ Karpukhin and D. Stern, {\em Min-max harmonic maps and a new characterization of conformal eigenvalues}, J.\ Eur.\ Math.\ Soc.\ \textbf{26}, no.~11, 4071--4129 (2024).

	\bibitem{ksII}
	M.\ Karpukhin and D. Stern, {\em Existence of harmonic maps and eigenvalue optimization in higher dimensions}, Invent.\ math.\ \textbf{236}, 713--778 (2024).



	\bibitem{SpecPart}
	J.\ B.\ Kennedy, P.\ Kurasov, C.\ Léna and D.\ Mugnolo, {\em A theory of spectral partitions of metric graphs}, Calc.\ Var.\ \textbf{60}, article no.~61 (2021).


	\bibitem{kkmm}
	J.\ B.\ Kennedy, P.\ Kurasov, G.\ Malenová and D.\ Mugnolo, {\em On the spectral gap of a quantum
			graph}, Ann.\ Henri Poincaré \textbf{7}, 2439--2473 (2016).


	\bibitem{kokarev}
	G.\ Kokarev, {\em Variational aspects of Laplace eigenvalues on Riemannian surfaces}, Adv.\ Math.\ \textbf{258}, 191--239 (2014).



	\bibitem{konic21}
	A.\ Kostenko and N.\ Nicolussi, {\em Laplacians on Infinite Graphs}, Memoirs Eur. Math. Soc. \textbf{3}, EMS Press, 2023.


	\bibitem{krein}
	M.\ G.\ Krein, {\em On certain problems on the maximum and minimum of characteristic values and on the Lyapunov zones of stability}, Amer.\ Math.\ Soc.\ Transl.\ Ser.\ 2(1), 163--187 (1955).


	\bibitem{kurasov2023spectral}
	P.\ Kurasov, {\em Spectral Geometry of Graphs}, Springer, Berlin, Heidelberg, 2023.

	\bibitem{lss}
	D.\ Lenz, M.\ Schmidt and P.\ Stollmann, {\em Topological Poincaré type inequalities and lower bounds on the infimum of the spectrum for graphs}, preprint, \arxiv{1801.09279} (2018).


	\bibitem{App: Comparison results}
	A.\ B.\ Mingarelli, {\em Volterra-Stieltjes integral equations and generalized ordinary differential expressions}, 1983rd ed., Berlin, Springer, 1983.

	\bibitem{pankrashkin}
	K.\ Pankrashkin, {\em Unitary dimension reduction for a class of self-adjoint extensions with applications to graph-like structures}, J.\ Math.\ Anal.\ Appl.\ \textbf{396}, no.~2, 640--655 (2012).

	\bibitem{prt}
	V.\ Pivovarchik, N.\ Rozhenko and C.\ Tretter, {\em Dirichlet–Neumann inverse spectral problem for a star graph of Stieltjes strings}, Linear Algebra Appl.\ \textbf{439}, no.~8, 2263--2292 (2013).

	\bibitem{post}
	O.\ Post, {\em Spectral Analysis on Graph-Like Spaces}, Lect.\ Notes in Math.\ {\bf 2039}, Springer-Verlag, Berlin, Heidelberg, 2012.


	\bibitem{vinokurov}
	D. Vinokurov, {\em Maximizing higher eigenvalues in dimensions three and above}, preprint, \arxiv{2506.09328} (2025).

	\bibitem{App: Volkmer}
	H.\ Volkmer, {\em Eigenvalue problems of Atkinson, Feller and Krein, and their mutual relationship}, Electronic Journal of Differential Equations \textbf{2005}, no.~48, 1--15 (2005).

	\bibitem{VonBelow}
	J.\ von Below, {\em A characteristic equation associated to an eigenvalue problem on $c^2$-networks}, Linear Algebra Appl.\ \textbf{71}, 309--325 (1985).


\end{thebibliography}
\end{document}